\documentclass[12pt]{amsart}
\usepackage{amsmath,amssymb,latexsym, amsthm, amscd, mathrsfs, stmaryrd, tikz, ulem}
\usepackage[linktocpage=true]{hyperref}
\usepackage{todonotes}
\usepackage{color}
\usepackage[all]{xy}

\hypersetup{colorlinks,linkcolor=red,urlcolor=cyan,citecolor=blue}

\newtheorem{thm}{Theorem} [section]
\theoremstyle{definition}

\newtheorem{rem}[thm]{Remark}
\theoremstyle{plain}
\newtheorem{prop}[thm]{Proposition}
\newtheorem{lem}[thm]{Lemma}

\numberwithin{equation}{section}

\newcommand{\al}{\alpha}

\newcommand{\Bl}{\mathcal B}
\newcommand{\C}{\mathbb C}
\newcommand{\D}{D(2|1;\zeta)}

\newcommand{\ep}{\epsilon}
\newcommand{\E}{\mathcal E}
\newcommand{\hf}{{\Small \frac12}}

\newcommand{\g}{\mathfrak{g}}
\newcommand{\gl}{\mathfrak{gl}}
\newcommand{\h}{\mathfrak{h}}
\newcommand{\hgt}{\mathrm{ht}}

\newcommand{\la}{\lambda}
\newcommand{\n}{\mathfrak n}
\newcommand{\N}{\mathbb N}
\newcommand{\one}{{\ov 1}}
\newcommand{\osp}{\mathfrak{osp}}
\newcommand{\om}{\omega}
\newcommand{\oo}{{\ov 0}}
\newcommand{\ov}{\overline}
\newcommand{\OO}{\mathcal O}

\newcommand{\pr}{\text{pr}}

\newcommand{\sll}{\mathfrak{sl}_2}

\newcommand{\wl}{\texttt{w}_\circ}
\newcommand{\Wt}{\text{WT}}
\newcommand{\Z}{\mathbb Z}

\newcommand{\LL}[1]{L_{{#1}}}
\newcommand{\M}[1]{M_{#1}}
\newcommand{\PP}[1]{P_{{#1}}}
\newcommand{\TT}[1]{T_{{#1}}}

\newcommand{\blue}[1]{{\color{blue}#1}}
\newcommand{\red}[1]{{\color{red}#1}}

\title[Character formulae in super category $\mathcal O$ for $G(3)$]{Character formulae in category $\mathcal O$ for exceptional Lie superalgebra $G(3)$ }

\author[Shun-Jen Cheng]{Shun-Jen Cheng}
\address{Institute of Mathematics, Academia Sinica, Taipei, Taiwan 10617} \email{chengsj@gate.sinica.edu.tw}

\author[Weiqiang Wang]{Weiqiang Wang}
\address{Department of Mathematics, University of Virginia, Charlottesville, VA 22904} \email{ww9c@virginia.edu}


\begin{document}

\maketitle

\begin{abstract}
We classify the blocks, compute the Verma flags of tilting and projective modules in the BGG category $\mathcal O$ for the exceptional Lie superalgebra $G(3)$. The projective injective modules in $\mathcal O$ are classified. We also compute the Jordan-H\"older multiplicities of the Verma modules in $\mathcal O$.
\end{abstract}

\maketitle

\setcounter{tocdepth}{1}
\tableofcontents

  \section{Introduction}

  \subsection{}

Among all the complex simple Lie superalgebras, there are 3 exceptional ones: $\D$ (where $\zeta\not=0,-1$ is a complex parameter), $G(3)$ and $F(3|1)$; cf. \cite{FK76} and \cite{CW12, Mu12}. While there are now complete solutions to the irreducible character problem in the BGG categories for the infinite series of basic Lie superalgebras  (see \cite{CLW11, CLW15, CFLW,  BLW17, BW13, Bao17}),  the BGG category of the exceptional Lie superalgebras were not investigated until very recently. In \cite{CW17} we initiated  the study of character formulae in the BGG category for exceptional Lie superalgebras by determining completely the Verma flags of tilting and projective modules in the BGG category for $\D$.
  \subsection{}
In this sequel to \cite{CW17}, we focus on the BGG category $\OO$ of modules of integral weights for  $G(3)$, the exceptional Lie superalgebra of dimension 31. The even subalgebra of $G(3)$ is $\g_\oo = G_2 \oplus \sll$ while its odd part is the tensor product of the 7-dimensional $G_2$-module with the 2-dimensional natural representation of $\sll$.
We classify the blocks, describe the Verma flags of tilting and projective modules, and classify the projective tilting modules in $\OO$. Then we compute the Jordan-H\"older (i.e., composition) multiplicities of Verma modules in $\OO$.

  \subsection{}

Our first basic result is a classification of the blocks in the BGG category $\OO$ for $G(3)$. The blocks in $\OO$ are divided into typical blocks (in which, roughly speaking, the super phenomena do not occur) and atypical blocks. The typical blocks were completely described by Gorelik's work \cite{Gor02a, Gor02b}. Indeed the typical blocks are controlled by the Weyl group $W$ of $G(3)$, and Gorelik shows that they are equivalent to some corresponding blocks of the even subalgebra $\g_\oo = G_2 \oplus \sll$.

The atypicality of a weight cannot be easily read off using the conventional notation of weights for $G(3)$. To remedy this, we introduce the notion of {\it symbols}, which are in bijection with weights; a symbol involves 3 coordinates on the $G_2$-part. The atypicality is readily read off from the symbol of a weight, and moreover, the action of the Weyl group $W_2$ of $G_2$ on the symbols is transparent too, as it is identified with (signed) permutations of $S_3$.

We classify the atypical blocks to be $\Bl_k$, parametrized by nonnegative integers $k$; see Theorem~\ref{thm:blocks}. Note that each atypical block $\Bl_k$ contains infinitely many simple objects, in contrast to typical blocks. The weight poset for each block $\Bl_k$ is acted upon by $W$, and we describe explicitly the set of anti-dominant weights $f_{k,n}$ in each poset (i.e., the $W$-transversals in each poset), which are parametrized by nonnegative integers $n$. The collection of $f_{k,n}^\sigma$, for all $k, n\in \N, \sigma \in W$, provides all the atypical integral weights.

  \subsection{}

The strategy of constructing the tilting modules in $\OO$ and obtaining their Verma flag multiplicities is as follows. We apply translation functors by tensoring with the 31-dimensional adjoint module. Our construction of tilting modules is inductive with respect to the Bruhat order of $W$, starting from the anti-dominant weights. A translation functor sends a tilting module to a direct sum of tilting modules. By choosing an initial tilting module properly, one hopes that after applying a translation functor one will get a candidate for a tilting module.

While the overall strategy is somewhat similar to the $\D$ case in \cite{CW17}, the two major steps for implementing it are much more difficult for $G(3)$ than for $\D$, as we explain below.
  \subsection{}

The first step is to produce candidates of Verma flags for tilting modules through translations, and the difficulty lies in the computational complexity for $G(3)$.
Our prsent work makes extensive use of Mathematica computations in both finding the suitable initial tilting modules and writing out the Verma flags of the resulting modules after translations. When  good choices of initial tilting modules are made, the patterns for Verma flags of (candidates of) tilting modules of highest weights ${f_{k,n}^\sigma}$  for varying $k,n$ and $\sigma$ become regular, and they can be summarized by concrete formulae via the Bruhat graph of $W_2$. Such formulae can be verified directly in principle.

Note $G_3$ has an even subalgebra $G_2 \oplus \sll$, and its Weyl group $W=A_1\times W_2$ has order 24. In contrast to $\D$, the even subalgebra $\D_\oo\cong\sll \oplus \sll \oplus \sll$ is a direct sum of the rank one simple Lie algebra, the Weyl group is $A_1\times A_1 \times A_1$, and hence the typical block structure is extremely simple. This makes it possible to compute by hand all the character formulae for $\D$, as we have done in \cite{CW17}. The passage from working with $A_1$ of rank one for $\D$ to $W_2$ of rank 2 for $G(3)$ increases the computational complexity dramatically.

\vspace{2mm}
The second main step is to establish the indecomposability of the candidate tilting modules, and this requires some new conceptual ideas  (note the lengths of Verma filtrations could be as long as 60). By construction a module obtained by applying a translation functor to an initial tilting module has a Verma flag such that the maximum weight among all the highest weights of the Verma modules in that flag is the desired highest weight of the tilting module in question.  Indecomposability would imply that the resulting module is indeed the desired tilting module.

Now Soergel duality relates the Verma flag multiplicity in a tilting module to the Jordan-H\"older multiplicity in a Verma module, and  for super BGG categories this was established by Brundan \cite{Br04}. Recall that, in the setting of $\D$ in \cite{CW17}, Soergel duality  when combined with some singular vector formulae therein is more or less sufficient to establish the desired indecomposability of the candidates for tilting modules, largely thanks to the simple nature of the even subalgebra of $\D$ and its Weyl group.

In our current $G(3)$ setting, the singular vector formulae associated to odd reflections are unknown, though a singular vector formula associated to even reflections has recently been established by Sale \cite{Sa17}. But even if all the singular vector formulae were found, for our purpose we would need to know further if various compositions of Verma module homomorphisms associated to even/odd reflections are nonzero, which is totally unclear due to the existence of zero divisors in the enveloping algebra of $G(3)$.

To overcome this difficulty we make use of the super Jantzen sum formula established by Gorelik \cite{Gor04}; also see Musson \cite{Mu12}. We derive from the super Jantzen sum formula a criterion (see Proposition~\ref{prop:flags}) for the existence of some composition factor in a Verma module, or equivalently by Soergel duality, for the appearance of some Verma module in a Verma flag of a tilting module. This criterion is valid for any basic Lie superalgebra and of independent interest, and it becomes extremely effective in our $G(3)$ setting because our candidates for tilting modules (with one exception) turn out to have 2 or 3 layers of Verma flags. (Here a layer means a $W$-orbit of highest weights.) Proposition~\ref{prop:flags} also plays a useful role in the classification of atypical blocks.

  \subsection{}

Our formulae for Verma flags of  tilting modules are rather uniform and compact (in contrast to the many separate cases for $\D$), and they are described via the Bruhat graph of $W_2$; see Theorems~\ref{thm:k=1+sigma}, \ref{thm:0:k=2+sigma}, and \ref{thm:0:k=1:sigma}. The formulae for tilting modules whose highest weights are singular or nearly singular  behave a little differently. The maximal multiplicity of a given Verma module in the Verma filtrations of tilting modules are shown to be 3, and this occurs only in a couple of tilting modules in the principal block $\Bl_0$.  The maximal length of Verma filtrations of tilting modules in $\OO$ turns out to be 60.

The principal block $\Bl_0$ requires some extra care; see Theorems~\ref{thm:k=0:sigma} and \ref{thm:0:k=0:sigma}. In spite of considerable efforts, 
there is still one particular tilting module in $\Bl_0$, whose Verma flag structure remains unknown to us; see Remark~\ref{rem:T20}. (We note a remarkable analogy with $\D$: there is a particular tilting module $T_{1,-1,-1}$ in the principal block of $\D$ \cite[Theorem~3.5]{CW17}, whose Verma flag structure was very difficult to pin down. The solution therein is based on precise information of zero-weight subspaces of some relevant modules.)
  \subsection{}

Via Soergel duality and BGG reciprocity, we convert the formulae for Verma flags of tilting modules into  formulae for Verma flags of projective modules. Moreover, we classify the projective tilting modules (= projective injective modules) in each atypical block $\Bl_k$; we show that they are exactly the tilting modules parametrized by the dominant integral weights. This is carried out in Section~\ref{sec:proj}.

By using BGG reciprocity we then convert in Section~\ref{sec:comp} our formulae for projective modules to obtain Jordan-H\"older multiplicity for Verma modules in $\Bl_k$, for $k\ge 1$.
  \subsection{}

The atypical blocks in the category of {\it finite-dimensional} $G(3)$-modules were classified and  the characters of projective modules in these atypical blocks by Germoni \cite{Ger00}; see also \cite{Ma14}. The character formulae of projective modules and simple modules in a parabolic BGG category for $G(3)$ were obtained in \cite{SZ16}.
  \subsection{}

This paper is organized as follows.
In Section~\ref{sec:basic}, we recall the super Jantzen sum formula. Then  we formulate in Proposition~\ref{prop:flags} a criterion
for the existence of some Verma module in a flag of a tilting module.

The notion of symbols is introduced and a weight-symbol correspondence is formulated in Section~\ref{sec:blocks}. We classify the atypical blocks in $\OO$, and also describe the weight poset for each atypical block $\Bl_k$, for $k\in \N$.

In Section~\ref{sec:charT}, we obtain the formulae for Verma flags of roughly half of the tilting modules in $\OO$, i.e., those with highest weights ${f_{k,n}^\sigma}$, for $\sigma \in W_2$.
We then obtain in Section~\ref{sec:charT0} the formulae for Verma flags of the remaining half of the tilting modules in $\OO$, i.e., those with highest weights ${f_{k,n}^\sigma}$, for $\sigma \in W \backslash W_2$.  Some cases in the blocks $\Bl_0$ and $\Bl_1$ need to be addressed separately.

We convert in Section~\ref{sec:proj} the formulae for Verma flags of  tilting modules to formulae for Verma flags of projective modules in $\OO$. We then classify the projective tilting modules in all blocks $\Bl_k$.
In Section~\ref{sec:comp}, we compute via BGG reciprocity the Jordan-H\"older multiplicities of Verma modules in $\Bl_k$, for $k\ge 1$.

 \vspace{.4cm}
 {\bf Acknowledgment.}
We acknowledge an extensive use of Mathematica, and our Mathematica code can be made available upon request. The first author is partially supported by a MoST and an Academia Sinica Investigator grant, while the second author is partially supported by an NSF grant DMS-1702254. We thank University of Virginia and Academia Sinica for hospitality and support.

\section{Conditions for nonzero Verma flag multiplicities in tilting modules}
    \label{sec:basic}

In this section, we work with an arbitrary basic Lie superalgebra $\g$. We give sufficient conditions for certain nonzero Jordan-H\"older multiplicity in a Verma $\g$-module, or equivalently, for certain nonzero Verma flag multiplicity in a tilting module.

Let $\Phi^+=\Phi^+_{\bar 0}\cup\Phi^+_{\bar 1}$ be a set of positive roots for a basic Lie superalgebra $\g$ (i.e., the Lie superalgebras of types $\gl, \osp, D(2|1;\zeta), F(3|1), G(3)$); cf. \cite{CW12, Mu12}. Then $\g$ admits a non-degenerate supersymmetric even bilinear form $(\cdot, \cdot)$, and $\Phi^+$ induces a triangular decomposition $\g =\n^- \oplus \h \oplus \n^+$. For $\alpha \in \Phi_\oo$, we denote by $\alpha^\vee \in \h$ the corresponding coroot so that
\[
\langle  \la,\alpha^\vee\rangle = 2(\la, \alpha)/(\alpha, \alpha),
\quad \forall \la\in X.
\]
The reflection $s_\alpha$ on $\h^*$, for $\alpha \in \Phi_\oo$,  is defined as usual by letting $s_\alpha (\la) =\la - \langle  \la,\alpha^\vee\rangle \alpha$.

Let $\OO$ be the BGG category which consists of $\g$-modules of integral weights that are locally finite with respect to $\mathfrak b=\h \oplus \n^+$. Denote by $M_\la$ the Verma module of highest weight $\la-\rho$ and denote by $\LL{\la}$ the unique simple quotient of $M_\la$, where $\rho$ is the Weyl vector associated to $\Phi^+$.
Denote by $\TT{\la}$ the tilting module of highest weight $\la-\rho$, which by definition is an indecomposable $\g$-module admitting a dual Verma flag and a Verma flag with highest term $\M{\la}$; denote by $(\TT{\la}:\M{\mu})$ the multiplicity of $\M{\mu}$ in a Verma flag of $\TT{\la}$. It is known that the projective cover $\PP{\la}$ of $\LL{\la}$ in $\OO$ also has a Verma flag, and we denote similarly by $(\PP{\la}: M_\mu)$ the multiplicity of $M_\mu$ in a Verma flag of $\PP{\la}$.
The BGG reciprocity and Soergel duality in $\OO$ are given as follows (cf. \cite{Br04}):
\begin{equation}  \label{tiltingD}
(\TT{-\la} : \M{-\mu})
=(\PP{\la}: M_\mu)
=[\M{\mu}: \LL{\la}], \qquad \text{ for }\la, \mu \in X.
\end{equation}

Define
\begin{align*}
\Phi^+_{\bar 1, \otimes}
 &=\{\alpha\in\Phi^+_{\bar 1}|(\alpha,\alpha)=0\},\qquad
\Phi^+_{\bar 1,\tikz\draw[fill] (0,0) circle (.4ex);}=\Phi^+_{\bar 1}\setminus\Phi^+_{\bar 1,\otimes},
\\
\Phi^+_{\bar 0,\tikz\draw (0,0) circle (.4ex);}
 &=\{\alpha\in\Phi^+_{\bar 0}|\alpha\not=2\beta,\beta\in\Phi^+_{\bar 1,\tikz\draw[fill] (0,0) circle (.4ex);}\}.
\end{align*}

The Jantzen sum formula has been generalized to basic Lie superalgebras by Gorelik and Musson as follows; see \cite{Gor04} and \cite[(10.3)]{Mu12}. Let $\N$ denote the set of nonnegative integers.

\begin{prop} [Super Jantzen sum formula]
 \label{prop:JSF}
There exists a finite filtration of $\g$-modules for the Verma module $M_\la$ (for $\la \in X$),
$
M_\la\supset M^1_\la\supset M^2_\la\supset \cdots,
$ 
such that
\begin{align}
  \label{jantzen:sum}
\begin{split}
\sum_{i>0} &{\rm ch} M^i_\la
\\
 = & \sum_{\alpha\in \Phi^+_{\bar 0,\tikz\draw (0,0) circle (.3ex);},\langle\la,\alpha^\vee\rangle\in\N\setminus\{0\}} {\rm ch}M_{s_\alpha\la}
 + \sum_{\beta\in \Phi^+_{\bar 0}\setminus\Phi^+_{\bar 0,\tikz\draw (0,0) circle (.3ex);},\langle\la,\beta^\vee\rangle\in\hf+\N} {\rm ch}M_{s_\beta\la}
+ \sum_{\gamma\in \Phi^+_{\bar 1,\otimes},(\la,\gamma)=0} \frac{{\rm ch}M_{\la-\gamma}}{1+e^{-\gamma}}.
\end{split}
\end{align}
\end{prop}
By \cite[Theorem~1.10]{Mu17}, the expression $\frac{{\rm ch}M_{\la-\gamma}}{1+e^{-\gamma}}$ above is the character of a genuine $\g$-module (which we shall denote by $M'_{\la-\gamma}$ and call it a {\it Musson} module).

We shall explore various implications of the Jantzen sum formula on the composition factors of Verma modules and then Verma flags of tilting modules.
The following proposition is of general interest, and it provides an essential tool in verifying various modules of Lie superalgebra $G(3)$ constructed in later sections are tilting modules.

\begin{prop}
  \label{prop:flags}
Retain the notations above for a basic Lie superalgebra $\g$.
Let $\la\in X$, $\alpha_i\in\Phi^+_{\bar 0}$, $1\le i\le k$, and $\beta,\gamma\in\Phi^+_{\bar 1}$. Let $w=s_{\alpha_k}\cdots s_{\alpha_2} s_{\alpha_1}\in W$.
\begin{itemize}
\item[(1)]
Suppose that $\langle\la,\alpha_1^\vee\rangle>0$. Then $(T_\la: M_{s_{\alpha_1} \la})>0$.

\item[(2)]
Suppose that
$\langle s_{\alpha_{i-1}}\cdots s_{\alpha_{1}}\la,\alpha_i^\vee\rangle > 0$,
for all $i=1,\ldots,k$. Then $(T_\la:M_{w\la})>0$.

\item[(3)]
Suppose that $(\la,\beta)=0$. Then $(T_\la:M_{\la-\beta})>0$.

\item[(4)]
Suppose that $(\la,\beta)=0$ and $\langle s_{\alpha_{i-1}}\cdots s_{\alpha_{1}}(\la-\beta),\alpha_i^\vee\rangle> 0$
 for all $i=1,\ldots,k$.
Then $(T_\la:M_{w(\la-\beta)})>0$.

\item[(5)]
Suppose that $(\la,\beta)=(\la-\beta,\gamma)=0$
and $\hgt(\beta)<\hgt(\gamma)$.  Then $(T_\la:M_{\la-\beta-\gamma})>0$.

\item[(6)]
Suppose that $(\la,\beta)=(\la-\beta,\gamma)=0$,
$\hgt(\beta)<\hgt(\gamma)$, $\langle s_{\alpha_{i-1}}\cdots s_{\alpha_{1}}(\la-\beta-\gamma),\alpha_i^\vee\rangle > 0$, for all $i=1,\ldots,k$.
Then $(T_\la:M_{w(\la-\beta-\gamma)})>0$.
\end{itemize}
\end{prop}

\begin{proof}
Parts (1) and (3) can be viewed as special cases of (2) and (4), but we have chosen to formulate them separately.
Parts (1) and (3) follow immediately by the Soergel duality \eqref{tiltingD} and the Jantzen sum formula \eqref{jantzen:sum}.
We shall give the details of proofs for (4)--(5) below.

To prove (4), by the duality \eqref{tiltingD}, it suffices to show that $[M_{-w(\la-\beta)}:L_{-\la}]>0$. We set $w_i =s_{\alpha_i}s_{\alpha_{i-1}}\cdots s_{\alpha_1}$ (note $w_0=1$, $w_k=w$), and $\mu_i=-w_i(\la-\beta)$, for $0\le i \le k$. By the assumptions we have
\begin{align*}
\langle \mu_i,\alpha_i^\vee\rangle&=-\langle w_i(\la-\beta),\alpha_i^\vee\rangle
\\
&=-\langle s_{\alpha_i}s_{\alpha_{i-1}}\cdots s_{\alpha_1}(\la-\beta),\alpha_i^\vee\rangle \\
&=\langle  s_{\alpha_{i-1}}\cdots s_{\alpha_1}(\la-\beta),\alpha_i^\vee\rangle >  0.
\end{align*}
It follows by \eqref{jantzen:sum} with $\la=\mu_i$ every composition factor of $M_{s_{\alpha_i}\mu_i} =M_{\mu_{i-1}}$ is a composition factor of $M_{\mu_i}$, for each $1\le i\le k$. Hence every composition factor of $M_{\mu_0} =M_{-\la+\beta}$ is a composition factor of $M_{\mu_k} =M_{-w(\la-\beta)}$. Since $(-\la+\beta,\beta)=(\la,\beta)=0$, $L_{-\la}$ is a composition factor of $M_{-\la+\beta}$ (by Part (3) and \eqref{tiltingD}). Thus, $L_{-\la}$ is a composition factor of $M_{-w(\la-\beta)}$, whence (4).

We now prove (5). It follows by Musson \cite[Theorem 1.10]{Mu17} that $\frac{{\rm ch}M_{-\la+\beta}}{1+e^{-\gamma}}$ is the character of the Musson module $M'_{-\la+\beta}$, a quotient module of $M_{-\la+\beta}$. By \eqref{jantzen:sum} we have
\begin{align}
  \label{J+M}
\sum_{i>0}{{\rm ch}M^i_{-\la+\beta+\gamma}}
= \frac{{\rm ch}M_{-\la+\beta}}{1+e^{-\gamma}}+\ldots
= {\rm ch}M'_{-\la+\beta} +\ldots
\end{align}
where $``+\ldots"$ denotes the character of some genuine $\g$-module.
It follows by definition that
\begin{align}
  \label{aux101}
{\rm ch}M'_{-\la+\beta}= {\rm ch}M_{-\la+\beta}- {\rm ch}M_{-\la+\beta-\gamma} + {\rm ch}M_{-\la+\beta-2\gamma} -\cdots.
\end{align}
Since $\hgt(\beta)<\hgt(\gamma)$, each Verma module $M_{-\la+\beta-k\gamma}$, for $k\ge 1$, has zero $(-\la)$-weight space.
Therefore by \eqref{aux101} 
we have $[M'_{-\la+\beta}:L_{-\la}] =[M_{-\la+\beta}:L_{-\la}]>0$ (here ``$>0$" follows by Part (3) and \eqref{tiltingD}). Hence we obtain by \eqref{J+M} that $[M_{-\la+\beta+\gamma}:L_{-\la}]>0$, which in turn implies by \eqref{tiltingD} that  $\left(T_\la:M_{\la-\beta-\gamma}\right)>0$, as claimed.

Part~(2) can be formally obtained from (4) and its proof by setting $\beta=0$. 
Part (6) is proved based on (5) in an entirely similar way as Part (4) (which was based on (3)). We skip these similar arguments.
\end{proof}

\begin{rem}
The condition $\hgt(\beta)<\hgt(\gamma)$ can be replaced by the weaker condition $\beta-\gamma \not \in \N \Phi^+$. For $\g=G(3)$ below, these two conditions are the same.
\end{rem}

 \section{Classification of blocks in the BGG category $\OO$ for $G(3)$}
   \label{sec:blocks}
 \subsection{Weights and roots for $G(3)$}

 From now on, we let $\g =\g_\oo \oplus \g_\one$ be the exceptional Lie superalgebra $G(3)$ over $\C$ with even subalgebra $\g_\oo =  G_2 \oplus \sll$.
We have $\g_\one \cong \underline{\bf 7} \otimes \underline{\bf 2}$ under the adjoint $\g_\oo$-action,
where $\underline{\bf 2}$ denotes the natural $\sll$-module and
$\underline{\bf 7}$ denotes the $7$-dimensional simple $G_2$-module.

To describe the roots for $G_2$ and $\g$, we introduce $\ep_1, \ep_2, \ep_3$ which satisfy the linear relation
 \[
 \ep_1 +\ep_2 +\ep_3 =0.
 \]
A bilinear form $(\cdot, \cdot)$ on
\[
X := \Z\delta \oplus \Z\ep_1 \oplus \Z\ep_2
\]
is given by
\[
(\delta, \delta) =-2, \quad (\delta, \ep_i) =0,
\quad
(\ep_i, \ep_i) =2, \quad (\ep_i, \ep_j) =-1, \quad \text{ for } 1\le i \neq j \le 3.
\]

We choose the simple system $\Pi =\{\alpha_1, \alpha_2, \alpha_3\}$ for $\g$, where
\[
\alpha_1 = \ep_2 -\ep_1,
\quad
\alpha_2=\ep_1,
\quad
\alpha_3 =  \delta +\ep_3.
\]
The Dynkin diagram associated to $\Pi$ is depicted as follows:
\vskip .5cm
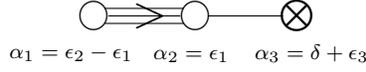
\begin{center}
\begin{tikzpicture}
\node at (0,0) {$\bigcirc$};
\draw (0.2,0)--(1.155,0);
\draw (0.15,0.1)--(1.2,0.1);
\draw (0.15,-0.1)--(1.2,-0.1);
\node at (1.35,0) {$\bigcirc$};
\node at (0.75,0) {\Large $>$};
\draw (1.52,0)--(2.52,0);
\node at (2.7,0) {$\bigotimes$};
\node at (-0.3,-.5) {\tiny $\al_1=\ep_2-\ep_1$};
\node at (1.3,-.53) {\tiny $\al_2=\ep_1$};
\node at (2.9,-.5) {\tiny $\al_3=\delta+\ep_3$};
\end{tikzpicture}
\end{center}

The root system of $\g$ is a union of even and odd roots: $\Phi =\Phi_\oo \cup \Phi_\one$.
The positive roots associated to $\Pi$ are $\Phi^+ =\Phi_\oo^+ \cup \Phi_\one^+$,
where
\[
\Phi_\oo^+ =\{2\delta, \ep_1, \ep_2, -\ep_3, \ep_2 -\ep_1, \ep_1 -\ep_3, \ep_2 -\ep_3 \},
\qquad
\Phi_\one^+ =\{\delta, \,   \delta \pm \ep_i\mid 1\le i \le 3\}.
\]
Then
\[
\Phi^+_{\bar 0,\tikz\draw (0,0) circle (.4ex);} =\{\ep_1, \ep_2, -\ep_3, \ep_2 -\ep_1, \ep_1 -\ep_3, \ep_2 -\ep_3 \},
\quad
\Phi^+_{\bar 1, \otimes} =\{\delta \pm \ep_i\mid 1\le i \le 3\},
\quad
\Phi^+_{\bar 1,\tikz\draw[fill] (0,0) circle (.4ex);}= \{\delta\}.
\]

The Weyl vector for $\g$ is defined to be $\rho = \rho_{\bar 0} -\rho_{\bar 1}$, where $\rho_{\bar 0} =\hf\sum_{\al \in \Phi_\oo^+} \al$ and $\rho_{\bar 1} =\hf \sum_{\beta \in \Phi_\one^+} \beta$. One computes that
\begin{equation}
  \label{rho}
 \rho =- \frac52\, \delta +2\ep_1 +3\ep_2, \qquad \rho_{\bar 1} =\frac{7}{2}\delta.
\end{equation}

Note that $\{\al_1, \al_2\}$ forms a simple system of $G_2$.
Denote by $\om_1$ and $\om_2$ the corresponding fundamental weights of $G_2$.
We have
\begin{align*}
\om_1 =\ep_1 + 2\ep_2, \qquad
\om_2 &=\ep_1 +\ep_2;
\\
\ep_1= 2\om_2 -\om_1,  \qquad
\ep_2 &= \om_1 -\om_2.
\end{align*}
Therefore, we can identify $X$ with the weight lattice of $\g$, and we have
\[
X =\Z\delta \oplus X_2,
\]
where
\[
X_2=\Z\om_1 \oplus \Z\om_2 =\Z\ep_1 \oplus \Z\ep_2
\]
is the weight lattice of $G_2$.

We shall denote by
\[
s_0=s_{2\delta}, \qquad s_1=s_{\alpha_1}, \qquad  s_2=s_{\alpha_2}.
\]
The Weyl group $W$ of $\g$ is
\[
W=A_1 \times W_2,
\]
where $W_2$ denotes the Weyl group of $G_2$ and $A_1 =\langle s_0 \rangle$ is the Weyl group of $\sll$.
Denote by $\wl$ the longest element in $W_2$. The Bruhat graph of the Weyl group $W_2$  is as follows (where $21$ is an abbreviation of $s_2s_1$ and so on):
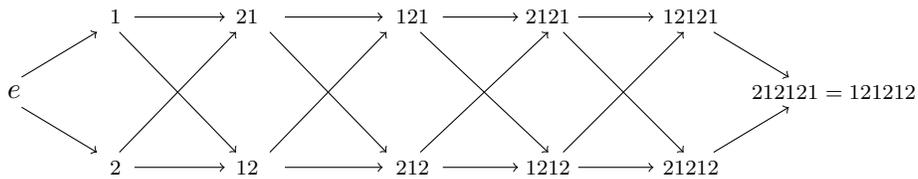
\begin{center}
\begin{tikzpicture}
\node at (-.1,0) {\small $e$};
\draw[->] (0,0.2)--(1,.8);
\draw[->](0,-0.2)--(1,-.8);
\node at (1.25,1) {\tiny $1$};
\node at (1.25,-1) {\tiny $2$};
\draw[->] (1.5,1)--(2.7,1);
\draw[->] (1.5,-1)--(2.7,-1);
\node at (3,1) {\tiny$21$};
\node at (3,-1) {\tiny$12$};
\draw[->] (1.3,0.8)--(2.85,-.8);
\draw[->] (1.3,-0.8)--(2.85,.8);
\draw[->] (3.5,1)--(4.8,1);
\draw[[->] (3.5,-1)--(4.8,-1);
\node at (5.2,1) {\tiny$121$};
\node at (5.2,-1) {\tiny$212$};
\draw[->] (3.3,0.8)--(4.85,-.8);
\draw[->] (3.3,-0.8)--(4.85,.8);
\draw[->] (5.6,1)--(6.6,1);
\draw[->] (5.6,-1)--(6.6,-1);
\node at (7,1) {\tiny$2121$};
\node at (7,-1) {\tiny$1212$};
\draw[->] (5.3,0.8)--(7,-.8);
\draw[->] (5.3,-0.8)--(7,.8);
\draw[->] (7.4,1)--(8.4,1);
\draw[->] (7.4,-1)--(8.4,-1);
\node at (8.9,1) {\tiny$12121$};
\node at (8.9,-1) {\tiny$21212$};
\draw[->] (7.2,0.8)--(8.8,-.8);
\draw[->] (7.2,-0.8)--(8.8,.8);
\draw[->] (9.2,0.8)--(10.2,.2);
\draw[->] (9.2,-0.8)--(10.2,-.2);
\node at (10.8,0) {\tiny $212121=121212$};
\end{tikzpicture}
\end{center}

The positive roots of $\g$ of equal heights are grouped together and listed in an increasing order from height 1 to 8 in the following table:

\begin{table}[h]
\caption{Heights of positive roots of $\g$}
\label{table:spinH}
\begin{tabular}{| c | c | c | c |c | c | c | c |}
\hline
1&  2& 3 & 4 &5&  6& 7 &8
\\
\hline
$\delta +\ep_3, \ep_1, \ep_2 -\ep_1$ &
$\delta -\ep_2, \ep_2$ &
$\delta -\ep_1, -\ep_3$ &
$\delta, \ep_1-\ep_3$ &
$\delta+\ep_1, \ep_2-\ep_3$ &
$\delta+\ep_2$ &
$\delta-\ep_3$ &
$2\delta$
\\
\hline
\end{tabular}
\newline
\smallskip
\end{table}

 \subsection{The symbols}


It is not obvious, by inspection, to decide whether a weight in the usual notation (as a linear combination of $\delta, \om_1, \om_2$ or a linear combination of $\delta, \ep_1, \ep_2$) is atypical and also whether two atypical weights are linked; cf. \cite{Ger00}. We shall introduce a different labeling convention for the weights, which we call {\it symbols}, that will make it transparent to observe whether or not a weight is atypical.

  A  $\rho$-shifted weight $\la=d \delta + a \om_1 +b \om_2 \in X+\rho$ satisfies $d \in \Z+\hf, a, b\in \Z$.
We define the scalars (or $\ep_i$-projections)
\[
x=\pr_{\ep_1} (\la), \quad
y =\pr_{\ep_2} (\la), \quad
z=\pr_{\ep_3} (\la)
\]
by writing $a \om_1 +b \om_2  =x\ep_1 + x' \ep_1^\perp = y \ep_2 + y' \ep_2^\perp = z \ep_3 + z' \ep_3^\perp$, where $\ep_i^\perp \in X_2$ is such that $(\ep_i, \ep_i^\perp) =0$.
 Actually we can choose  $\ep_1^\perp =\hf \om_1 =\hf \ep_1 +\ep_2$, $\ep_2^\perp= \ep_1 + \hf \ep_2$ and so on.
A direct computation shows $z=-x-y$.
We call $[d|x, y, z]$ the {\it symbol} corresponding to $\la$, and denote $\la \sim [d~|~x, y, z]$.
Introduce the set of symbols
 \begin{align}
   \label{eq:sym}
 \mathfrak{Symb} =\Big\{ [d~|~ x,y,-x-y], \text{ for } d  \in \Z +\hf,\; x,y \in \hf\Z,\;  3 \big |2(y-x) \Big\}.
 \end{align}

 \begin{lem}
We have a bijection $\Psi: X+\rho \longrightarrow \mathfrak{Symb}$, which sends
$\la=d \delta + a \om_1 +b \om_2$ to its symbol $[d ~|~\pr_{\ep_1} (\la), \pr_{\ep_2} (\la), \pr_{\ep_3}(\la)]$.
\end{lem}

\begin{proof}
 The lemma follows by a direct computation; indeed, the map $\Psi$ and its inverse map $\Psi^{-1}$ are given by the following formulae:
 \begin{align}
   \label{symbol}
    \begin{split}
 \Psi ( d \delta + a \om_1 +b \om_2)  &= [d~|~ b/2, (3a+b)/2, -(3a+2b)/2],
  \\
\Psi^{-1}( [d~|~ x,y,z]) &=  d \delta + \frac23(y-x) \om_1+ 2x \om_2,
    \end{split}
 \end{align}
 where $d \in \Z +\hf, a, b \in \Z, x,y \in \hf\Z,\;  3 \big |2(y-x).$
 \end{proof}

The symbols of $\ep_i, \ep_i-\ep_j$ are given by
\begin{align*}
\ep_1 \sim [0 ~|~ 1,\;  -1/2,\;  -1/2], & \quad \ep_2 -\ep_1 \sim [0 ~|~{-3/2},\;  3/2,\;  0],
\\
\ep_2 \sim [0 ~|~  {-1/2},\;  1,\;  -1/2], & \quad \ep_1 -\ep_3 \sim [0 ~|~ 3/2,\;  0,\;  -3/2],
\\
\ep_3 \sim [0 ~|~ {-1/2},\;  -1/2, \;  1], & \quad \ep_2 -\ep_3 \sim [0 ~|~ 0,\;  3/2,\;  -3/2].
\end{align*}
Here are the symbols of some odd roots:
\begin{align*}
&\pm\delta +\ep_1 \sim [\pm 1 ~|~ 1, -1/2, -1/2],
\\
&\pm\delta +\ep_2 \sim   [\pm   1 ~|~ {-1/2}, 1, -1/2],
\\
& \pm\delta +\ep_3 \sim [\pm 1 ~|~ {-1/2}, -1/2, 1].
\end{align*}

A weight $\la \in X+\rho$ is {\it atypical} if $(\la, \beta)=0$ for some odd root $\beta =\delta \pm\ep_i$, where $1\le i \le 3$; we will also say its symbol is atypical.

The Weyl group $W=A_1 \times W_2$ acts on $X+\rho=(\hf+\Z) \delta \oplus X_2$, where $s_0 \in A_1$ acts by sign change on $\delta$ and $W_2$ acts on $X_2$ by permuting the 3 indices in $\{\pm \ep_1, \pm \ep_2, \pm \ep_3\}$  (possibly coupled with a simultaneous sign change). More precisely, the simple reflections associated to $\al_1, \al_2$ act as follows:
\begin{align}
 \label{simple ref}
 \begin{split}
s_1 & : \; \ep_1 \mapsto \ep_2, \quad  \ep_2 \mapsto \ep_1, \quad \ep_3 \mapsto \ep_3,
\\
s_2& : \; \ep_1 \mapsto -\ep_1,\quad  \ep_2 \mapsto -\ep_3, \quad \ep_3 \mapsto -\ep_2.
\end{split}
\end{align}
By the formulae \eqref{symbol} for the bijection $\Psi$ we can verify directly that if $\la \sim [d | x, y, z]$ and $w=(\sigma, w_2) \in A_1 \times W_2$, then $w(\la) \sim [\sigma(d)~ |~ w_2(x,y,z)]$, where $W_2$ acts on the hyperplane $x+y+z=0$
by permuting the 3 coordinates (coupled with a possible simultaneous sign change); more explicitly, we have
$s_1(x,y,z) =(y,x,z)$, $s_2(x,y,z)=(-x,-z,-y)$.

\begin{prop}
  \label{prop:symbolgood}
A symbol $[d ~|~ x, y, z]$ is atypical if and only if $d$ (or $-d$) is equal to $x, y$ or $z$. Moreover, the weight-symbol correspondence $\Psi$ commutes with the natural Weyl group actions.
\end{prop}

\begin{proof}
The first statement follows by definition of the symbols. The second statement is a summary of the above discussions.
\end{proof}

\subsection{Atypical integral weights}

We now classify all atypical weights in $X$ via symbols and then group them into blocks.

A $\rho$-shifted weight $\la \in X+\rho$ is {\it anti-dominant} if $\langle \la, \alpha^\vee \rangle \le 0$, for all $\alpha \in \Phi^+_\oo$; we will call the corresponding symbol $f_\la$ anti-dominant.

Introduce the following atypical anti-dominant symbols,
%
for $k,n\in\N$:
\begin{align*}
f_{k,n} &=\big[-(2n+1)/2 ~\big |~ -(2n+1)/2,\; (n-3k -1)/2,\;  (3k+n+2)/2 \big], \quad 0\le n \le k-1,
\\
f_{k,k} &=\big[-(2k +1)/2 ~\big |~ -(2k +1)/2,\; -(2k +1)/2,\; 2k +1 \big],
\\
f_{k,n} &=\big[-(2n+1)/2 ~\big |~  (n-3k-1)/2,\; -(2n+1)/2,\; (3k+n+2)/2 \big], \quad  k+1\le n \le 3k,
\\
f_{k,3k+1}   &=\big[-3(2k +1)/2 ~\big |~ 0,\; -3(2k +1)/2,\; 3(2k +1)/2 \big],
\\
f_{k,n}&=\big[ -(2n+1)/2 ~\big |~ -(n-3k-1)/2,\; -(3k+n+2)/2,\; (2n+1)/2 \big], \quad n\ge 3k+2.
\end{align*}
The corresponding weights $\Psi^{-1}(f_{k,n})$ in $X+\rho$ are as follows:
\begin{align*}
\Psi^{-1}(f_{k,n}) &=-\frac{2n+1}{2}\delta - (n+k+1)\ep_1-(2k+1)\ep_2, \quad 0\le n \le k-1,
\\
\Psi^{-1}(f_{k,k}) &=-\frac{2k+1}{2}\delta - (2k+1)\ep_1-(2k+1)\ep_2,
\\
\Psi^{-1}(f_{k,n}) &=-\frac{2n+1}{2}\delta -(2k+1)\ep_1 - (n+k+1)\ep_2, \quad  k+1\le n \le 3k,
\\
\Psi^{-1}(f_{k,3k+1})   &=-\frac{6k+3}{2}\delta - (2k+1)\ep_1-(4k+2)\ep_2,
\\
\Psi^{-1}(f_{k,n})&=-\frac{2n+1}{2}\delta - (n-k)\ep_1-(n+k+1)\ep_2, \quad n\ge 3k+2.
\end{align*}

Note $f_{k,k}, f_{k,3k+1}$ are singular, and all the others are regular.
Both the stabilizer $\langle s_{1} \rangle$ of $f_{k,k}$ and the stabilizer $\langle s_{2} \rangle$ of $f_{k,3k+1}$ have order 2.
For all $k\ge 0$ and $n\ge 0$, we let
$f_{k,n}^\sigma$ be the image of $f_{k,n}$ under the action of the Weyl group element $\sigma \in W$.

\newpage

{\quad}
\vspace{2mm}

\begin{table}[h]
\caption{Anti-dominant weights $f_{k,n}$, for $k=0,1,2,3$, and $n\le 8$.}
\begin{center}
\begin{tikzpicture}%
\draw  (-1.5,1) --  (13.5,1);
\node [red] at (0,0) {\small $[-\hf|-\hf,-\hf,1]$};
\draw [magenta, thick] [->] (0,-0.4)--(0,-1.2);
\node [blue] at (0,-1.5) {\small $[-\frac{3}{2}|0,-\frac{3}{2},\frac{3}{2}]$};
\draw [cyan, thick] [->] (0,-1.9)--(0,-2.7);
\node at (0,-3) {\small $[-\frac{5}{2}|-\hf,-2,\frac{5}{2}]$};
\draw [cyan, thick] [->] (0,-3.4)--(0,-4.2);
\node at (0,-4.5) {\small $[-\frac{7}{2}|-1,-\frac{5}{2},\frac{7}{2}]$};
\draw [cyan, thick] [->] (0,-4.9)--(0,-5.7);
\node at (0,-6) {\small $[-\frac{9}{2}|-\frac{3}{2},-3,\frac{9}{2}]$};
\draw [cyan, thick] [->] (0,-6.4)--(0,-7.2);
\node at (0,-7.5) {\small $[-\frac{11}{2}|-2,-\frac{7}{2},\frac{11}{2}]$};
\draw [cyan, thick] [->] (0,-7.9)--(0,-8.7);
\node at (0,-9) {\small $[-\frac{13}{2}|-\frac{5}{2},-4,\frac{13}{2}]$};
\draw [cyan, thick] [->] (0,-9.4)--(0,-10.2);
\node at (0,-10.5) {\small $[-\frac{15}{2}|-3,-\frac{9}{2},\frac{15}{2}]$};
\draw [cyan, thick] [->] (0,-10.9)--(0,-11.7);
\node at (0,-12)  {\small $[-\frac{17}{2}|-\frac72,-5,\frac{17}{2}]$};
\draw [cyan, thick] [->] (0,-12.4)--(0,-13.2);
\node at (0,-13.5) {\small $\vdots$};
\node [below] at (0,-13.8) {$k=0$};
\node at (4,0) {\small $[-\hf|-\hf,-2,\frac{5}{2}]$};
\draw [thick, lightgray] [->] (4,-0.4)--(4,-1.2);
\node [red] at (4,-1.5) {\small $[-\frac{3}{2}|-\frac{3}{2},-\frac{3}{2},3]$};
\draw [magenta, thick] [->] (4,-1.9)--(4,-2.7);
\node at (4,-3) {\small $[-\frac{5}{2}|-1,-\frac{5}{2},\frac{7}{2}]$};
\draw [magenta, thick] [->] (4,-3.4)--(4,-4.2);
\node at (4,-4.5) {\small $[-\frac{7}{2}|-\hf,-\frac{7}{2},4]$};
\draw [magenta, thick] [->] (4,-4.9)--(4,-5.7);
\node [blue] at (4,-6) {\small $[-\frac{9}{2}|0,-\frac{9}{2},\frac{9}{2}]$};
\draw [cyan, thick] [->] (4,-6.4)--(4,-7.2);
\node at (4,-7.5) {\small $[-\frac{11}{2}|-\frac{1}{2},-5,\frac{11}{2}]$};
\draw [cyan, thick] [->] (4,-7.9)--(4,-8.7);
\node at (4,-9) {\small $[-\frac{13}{2}|-1,-\frac{11}{2},\frac{13}{2}]$};
\draw [cyan, thick] [->] (4,-9.4)--(4,-10.2);
\node at (4,-10.5) {\small $[-\frac{15}{2}|-\frac{3}{2},-6,\frac{15}{2}]$};
\draw [cyan, thick] [->] (4,-10.9)--(4,-11.7);
\node at (4,-12) {\small $[-\frac{17}{2}|-2,-\frac{13}2,\frac{17}{2}]$};
\draw [cyan, thick] [->] (4,-12.4)--(4,-13.2);
\node at (4,-13.5) {\small $\vdots$};
\node [below] at (4,-13.8) {$k=1$};
\node at (8,0) {\small $[-\hf|-\hf,-\frac{7}{2},4]$};
\draw [thick, lightgray] [->] (8,-0.4)--(8,-1.2);
\node at (8,-1.5) {\small $[-\frac{3}{2}|-\frac{3}{2},-3,\frac{9}{2}]$};
\draw [thick, lightgray] [->] (8,-1.9)--(8,-2.7);
\node [red] at (8,-3) {\small $[-\frac{5}{2}|-\frac{5}{2},-\frac{5}{2},5]$};
\draw [magenta, thick] [->] (8,-3.4)--(8,-4.2);
\node at (8,-4.5) {\small $[-\frac{7}{2}|-2,-\frac{7}{2},\frac{11}{2}]$};
\draw [magenta, thick] [->] (8,-4.9)--(8,-5.7);
\node at (8,-6) {\small $[-\frac{9}{2}|-\frac{3}{2},-\frac{9}{2},6]$};
\draw [magenta, thick] [->] (8,-6.4)--(8,-7.2);
\node at (8,-7.5) {\small $[-\frac{11}{2}|-1,-\frac{11}{2},\frac{13}{2}]$};
\draw [magenta, thick] [->] (8,-7.9)--(8,-8.7);
\node at (8,-9) {\small $[-\frac{13}{2}|-\frac{1}{2},-\frac{13}{2},7]$};
\draw [magenta, thick] [->] (8,-9.4)--(8,-10.2);
\node [blue] at (8,-10.5) {\small $[-\frac{15}{2}|0,-\frac{15}{2},\frac{15}{2}]$};
\draw [cyan, thick] [->] (8,-10.9)--(8,-11.7);
\node at (8,-12) {\small $[-\frac{17}{2}|-\hf,-8,\frac{17}{2}]$};
\draw [cyan, thick] [->] (8,-12.4)--(8,-13.2);
\node at (8,-13.5) {\small $\vdots$};
\node [below] at (8,-13.8) {$k=2$};

\node at (12,0) {\small $[-\hf|-\hf,-5,\frac{11}{2}]$};
\draw [thick, lightgray] [->] (12,-0.4)--(12,-1.2);
\node at (12,-1.5) {\small $[-\frac{3}{2}|-\frac{3}{2},-\frac{9}{2},6]$};
\draw [thick, lightgray] [->] (12,-1.9)--(12,-2.7);
\node at (12,-3) {\small $[-\frac{5}{2}|-\frac{5}{2},-4,\frac{13}{2}]$};
\draw [thick, lightgray] [->] (12,-3.4)--(12,-4.2);
\node [red] at (12,-4.5) {\small $[-\frac{7}{2}|-\frac{7}{2},-\frac{7}{2},7]$};
\draw [magenta, thick] [->] (12,-4.9)--(12,-5.7);
\node at (12,-6) {\small $[-\frac{9}{2}|-3,-\frac{9}{2},\frac{15}{2}]$};
\draw [magenta, thick] [->] (12,-6.4)--(12,-7.2);
\node at (12,-7.5) {\small $[-\frac{11}{2}|-\frac{5}{2},-\frac{11}{2},8]$};
\draw [magenta, thick] [->] (12,-7.9)--(12,-8.7);
\node at (12,-9) {\small $[-\frac{13}{2}|-2,-\frac{13}{2},\frac{17}{2}]$};
\draw [magenta, thick] [->] (12,-9.4)--(12,-10.2);
\node at (12,-10.5) {\small $[-\frac{15}{2}|-\frac{3}{2},-\frac{15}{2},9]$};
\draw [magenta, thick] [->] (12,-10.9)--(12,-11.7);
\node at (12,-12)  {\small $[-\frac{17}{2}|-1,-\frac{17}{2},\frac{19}{2}]$};
\draw [magenta, thick] [->] (12,-12.4)--(12,-13.2);
\node at (12,-13.5) {\small $\vdots$};
\node [below] at (12,-13.8) {$k=3$};
\draw  (-1.5,-15) --  (13.5,-15);

\node at (2,-16) {$[f]$};
\draw [thick, lightgray] [->] (2,-16.3)--(2,-17.1);
\node at (2,-17.4) {$[g]$};
\node [below] at (2,-17.7) {$f-g=\delta+\epsilon_1$};
\node at (6,-16) {$[f]$};
\draw[thick, magenta] [->] (6,-16.3)--(6,-17.1);
\node at (6,-17.4) {$[g]$};
\node [below] at (6,-17.7) {$f-g=\delta+\epsilon_2$};
\node at (10,-16) {$[f]$};
\draw[thick, cyan] [->] (10,-16.3)--(10,-17.1);
\node at (10,-17.4) {$[g]$};
\node [below] at (10,-17.7) {$f-g=\delta-\epsilon_3$};
\end{tikzpicture}
\end{center}
\end{table}

\newpage
For $\sigma \in W$ and the corresponding left coset $\ov{\sigma} \in W/\langle s_{\al_i}\rangle$, for $i=1,2$, we set
 \begin{align*}
 f_{k,k}^{\ov{\sigma}} \stackrel{\text{def}}{=}
  f_{k,k}^{\sigma s_{\al_1}} = f_{k,k}^{\sigma},
\qquad
f_{k,3k+1}^{\ov{\sigma}} \stackrel{\text{def}}{=}
  f_{k,3k+1}^{\sigma s_{\al_2}} = f_{k,3k+1}^{\sigma},
\qquad \text{ for } \sigma \in W.
 \end{align*}
For $k\in \N$, we define
 \begin{align}
  \label{WT}
  \begin{split}
 \Wt_k^{\text{reg}} & = \big\{ f_{k,n}^\sigma ~\big |~ n \in \N \backslash \{k,3k+1\}, \sigma \in W \big\},
 \\
 \Wt_k^{\text{sing}} &= \big\{  f_{k,k}^{\ov{\sigma}} ~\big |~ \ov{\sigma} \in W/\langle s_{\al_1}\rangle \big \}
 \cup \big\{  f_{k,3k+1}^{\ov{\sigma}} ~\big |~ \ov{\sigma} \in W/\langle s_{\al_2}\rangle \big \},
 \\
 \Wt_k &= \Wt_k^{\text{reg}} \cup \Wt_k^{\text{sing}}.
 \end{split}
 \end{align}

 \begin{lem}
   \label{lem:aty:symbol}
Under the weight-symbol correspondence, the set of atypical $\rho$-shifted weights in $X+\rho$ is identified with the set of symbols $\sqcup_{k\in \N} \Wt_k$.
\end{lem}

\begin{proof}
Let $f$ be a symbol of an atypical $\rho$-shifted weight in $X+\rho$. By Proposition~\ref{prop:symbolgood}, up to $W$-conjugacy, we may assume $f$ is anti-dominant, i.e.,
\begin{align*}
f=[x|x,y,z], \text{ or } f=[y|x,y,z], \text{ or } f=[-z|x,y,z],
\end{align*}
where $y\le x\le 0$. Thus, the conditions on the symbols in \eqref{eq:sym} read in our setting as
\begin{align}
  \label{eq:conditions}
x\le 0, \qquad
x+y+z=0, \qquad
2(y-x)=3\ell\le 0,\; \text{ for some } \ell \in \Z.
\end{align}

Suppose first that $f=[x|x,y,z]$. Write $x=-n-\hf$ for $n\in \N$. Taking $\ell=k-n$, for $0\le n\le k$, we obtain $f=f_{k,n}$, for $0\le n\le k$.

Next suppose that $f= [y|x,y,z]$. Write $y=-n-\hf$ for $n\in \N$. Taking $\ell=-n-k$, we obtain $f=f_{k,n}$, for $k+1\le n\le 3k+1$.

Finally, suppose that $f= [-z|x,y,z]$. Write $z=n+\hf$ for $n\in \N$. The conditions in \eqref{eq:conditions} imply that $\ell$ is odd. Now, we let $-\ell=2k+1$ and we get $f=f_{k,n}$ for $n\ge 3k+2$.

Hence every atypical symbol is a $W$-conjugate of an anti-dominant symbol of the form $f_{k,n}$, for some $k,n\ge 0$. The lemma is proved.
\end{proof}

Under the weight-symbol bijection $\la \leftrightarrow f$, we shall write
\[
\TT{f} =\TT{\la}, \quad \PP{f} =\PP{\la}, \quad \M{f} =\M{\la}, \quad \LL{f} =\LL{\la}.
\]
\subsection{Classification of blocks}

The blocks in the category $\OO$ are divided into typical and atypical blocks. By definition, a (or any) simple module in a {\it typical} block has a typical highest weight; otherwise the block is called {atypical}.
\subsubsection{Typical blocks}

The typical blocks in $\OO$ are completely described by Gorelik (see \cite[Section 8.1.5]{Gor02a}) and \cite[Theorem 1.3.1]{Gor02b}).

\begin{prop}  [Gorelik]
\label{prop:typ-block}
Any typical block in $\OO$ is equivalent to a block in the BGG category  of $\g_\oo$-modules of integral weights.
\end{prop}
In particular, a typical block of $\g$-modules in $\OO$ contains finitely many simple modules whose highest weights (after a $\rho$-shift) lie in a  $W$-orbit. Under Gorelik's typical equivalence, the Weyl group orbits are isomorphic posets (and they are simultaneously regular or singular).

We shall adopt the following notation to record the Verma flag structure of a tilting module:
\[
\TT{f} =\sum_{g} c_{fg} \M{g},
\]
where $c_{fg}=(T_f:M_g)$ is the Verma flag multiplicity of the Verma module $\M{g}$ in $\TT{f}$.

\begin{lem}
   \label{lem:KL:sing}
Let $\la \in X+\rho$ be a typical anti-dominant singular weight. Let $W^\la$ be the set of minimal length left coset representatives of $W/W_\la$, where $W_\la=\{w\in W\vert w\la=\la\}$. Let $\sigma\in W^\la$. Then we have the following Verma flag for $T_{\sigma \la}$, for $\sigma\in W^\la$:
\begin{align}
  \label{eq:singular}
T_{\sigma \la} = \sum_{\tau\leq\sigma, \tau\in W^\la} M_{\tau \la}.
\end{align}
\end{lem}

\begin{proof}
Since $\la =d \delta +a\om_1 +b\om_2$ with $d\in \hf +\Z$, we have $\{e\} \neq W_\la \subseteq W_2$. The central character corresponding to the integral weight $\la-\rho$ is generic and hence strongly typical in the sense of Gorelik (see \cite[Section 8.1.5]{Gor02a}). Thus, by \cite[Theorem 1.3.1]{Gor02b} there is an equivalence of categories between the block containing the simple module $L_\la$, and a corresponding singular integral block of $\g_\oo$-modules, where we recall $\g_\oo = G_2 \oplus \sll$.

Note that the action of the Weyl group $A_1$ of $\sll$ on $\la$ is always regular. Hence the proof of the lemma is reduced to verifying the counterpart of \eqref{eq:singular} for a singular integral block of $G_2$-modules. Since $W_2$ is a dihedral group, the Kazhdan-Lusztig polynomials in this case are well known to be monomials. The $G_2$-singular block counterpart of \eqref{eq:singular} follows from this fact and \cite[Theorem 3.11.4(ii)(iv)]{BGS}.
%
\end{proof}

\subsubsection{Atypical blocks}

We now classify the atypical blocks in the category $\OO$ and hence complete the classification of blocks in $\OO$. Recall $\Wt_k$ from \eqref{WT}.

\begin{thm}
  \label{thm:blocks}
For each $k\in \N$,
there is a block $\Bl_k$ in $\OO$ which consists of modules with composition factors of the form $\LL{f}$, for $f \in \Wt_k$.
Moreover, any atypical block in $\OO$ is one of $\Bl_k$, for $k\in \N$.
\end{thm}

\begin{proof}
We identify the algebra $S(\h^*)$ with  $\C[\delta,\ep_1,\ep_2,\ep_3]/(\sum_{i=1}^3\ep_i)$. Under this identification it is well known (see, e.g., \cite[0.6.7]{Serg}) that the image of the Harish-Chandra homomorphism contains a subalgebra of $S(\h^*)$ of the form $P \cdot S(\h^*)^W$, where
$P=\prod_{i=1}^3(\delta^2-\ep_i^2).$ Since $(f,P)\not=0$ and $(g,P)=0$, for $f$ typical and $g$ atypical, we see that typical and atypical blocks cannot be linked.

On the other hand, if $f$ and $f'$ are both typical and $W f\cap W f'=\emptyset$, then either $(f,P)\not=(f',P)$, or else we can find an element in $P \cdot S(\h^*)^W$ separating $f$ and $f'$. Thus, two distinct typical blocks cannot be linked.

Recall from Lemma~\ref{lem:aty:symbol} the classification of atypical weights and atypical symbols $\{f^\sigma_{k,n}\}$.
For $\sigma\in W$, we compute the eigenvalue of the Casimir operator on the Verma module $M_{f^\sigma_{k,n}}$ to be
\begin{align*}
(f^\sigma_{k,n}+\rho,f^\sigma_{k,n}-\rho)=6k^2+6k.
\end{align*}
This shows that the Casimir element in $U(\g)$ separates different subcategories $\Bl_k$, for  $k\ge 0$.
Now, for each $k\ge 0$, it follows by Proposition~\ref{prop:flags} that the subcategory $\Bl_k$ is indeed indecomposable.

The theorem is proved.
\end{proof}

The block $\Bl_0$ is the principal block.
Since the posets $\Wt_k$ are non-isomorphic for different $k$, the blocks $\Bl_k$ are inequivalent as highest weight categories.

When it is clear from the context that we are dealing with a block $\Bl_k$ with a given $k$,
we shall omit the index $k$ by setting
 \[f_{n}^\sigma =f_{k,n}^\sigma.
 \]
Similarly, we shall denote the titling modules, projective covers, Verma modules, and the simple modules by dropping the index $k$ when there is no confusion on the underlying block $\Bl_k$:
\begin{align*}
\TT{n}^\sigma = \TT{k,n}^\sigma = \TT{ f_{k,n}^\sigma},  & \qquad
\PP{n}^\sigma = \PP{k,n}^\sigma = \PP{ f_{k,n}^\sigma},
\\
\M{n}^\sigma = \M{k,n}^\sigma = \M{ f_{k,n}^\sigma}, & \qquad
\LL{n}^\sigma =  \LL{k,n}^\sigma = \LL{ f_{k,n}^\sigma}.
\end{align*}
For the modules with singular weights in the two $W$-orbits, we often denote them in red and blue colors, e.g., $\TT{\red{k}}^\sigma, \M{\red{k}}^\sigma, \TT{\blue{3k+1}}^\sigma, \M{\blue{3k+1}}^\sigma$, and so on. The colors are helpful but lack of colors will not lead to any ambiguity of notations.

\begin{rem}
The atypical dominant integral weights and the atypical blocks in the category of {\it finite-dimensional} $\g$-modules were classified by Germoni  \cite[Theorem~4.1.1]{Ger00}.
\end{rem}

\section{Character formulae for tilting modules in $\OO$, I}
  \label{sec:charT}

In this section, we provide formulae for Verma flags of tilting modules $T^{\sigma}_n$, for $\sigma \in W_2$, in all blocks $\Bl_k$.

\subsection{Formulae for $\TT{n}^\sigma$ in the blocks $\Bl_k$}

Recall the Weyl group $W_2$ of $G_2$ with Bruhat ordering $\leq$. Let $\sigma_1,\sigma_2\in W_2$. A (Bruhat) {\it interval} of $W_2$ is a subset of $W_2$ of the form
\[
[\sigma_1, \sigma_2] :=\{\sigma \in W_2 \mid \sigma_1 \le \sigma \le \sigma_2\}.
\]
Let $[e,\sigma]/\langle s_i\rangle$ denote the subset of $W_2/\langle s_i\rangle$ which consists of left cosets of $[e,\sigma]$, for $i=1,2$. As before, when we write elements in $W_2$ we simply use their corresponding words in $\{1,2\}$; for example, $121$ means $s_1s_2s_1$ and $\sigma 2$ means $\sigma s_2$, and so on. Denote by $\ell (\sigma)$ the length of a reduced word for $\sigma \in W_2$.

We also write $\TT{f} =\sum_{g \in S} \M{g} + \sum_{h \in S'} \M{h} +\cdots$, where $g$ runs over certain sets $S, S'$ and it is possible for $\M{g}$ and $\M{h}$ to coincide for seemingly different $g,h$ (for example, they may be different coset representatives for singular weights). For a set (or a multiset) $D \subseteq W_2$ (or $D\subseteq W$ in Section~\ref{sec:charT0}), we shall introduce a shorthand notation
\begin{align}
  \label{MnD}
\M{n}^D =\sum_{\tau\in D} \M{n}^\tau.
\end{align}

Below we describe the Verma flags for roughly half the tilting modules in $\OO$, i.e., of the form $\TT{n}^\sigma$, for $\sigma \in W_2$.

\begin{thm}
  \label{thm:k=1+sigma}
The following formulae hold for tilting modules in the block $\Bl_k$: for $\sigma\in W_2$,
\begin{enumerate}
\item
$\TT{n}^{\sigma} =  \M{n}^{[e,\sigma]}  +\M{n+1}^{[e,\sigma]}, \quad \forall n \in \N \backslash \{k-1, k, 3k, 3k+1\}, \qquad (k\ge 0).$
\item
$\TT{\blue{3k+1}}^{\sigma } =\TT{\blue{3k+1}}^{\sigma 2}
=  \M{\blue{3k+1}}^{[e,\sigma]/\langle s_2\rangle} +\M{3k+2}^{[e,\sigma ]}, \quad
\text{ if } \ell(\sigma) >\ell(\sigma 2),
\qquad (k\ge 0).$
\item
$\TT{\red{k}}^{\sigma } =\TT{\red{k}}^{\sigma 1}
=  \M{\red{k}}^{[e,\sigma]/\langle s_1\rangle} +\M{k+1}^{[e,\sigma ]}, \qquad\quad
\text{ if } \ell(\sigma)> \ell(\sigma 1), \qquad (k\ge 1).$
\item
$\TT{k-1}^{\sigma} =
\begin{cases}
\M{k-1}^{[e,\sigma]} +  \M{\red{k}}^{[e,\sigma]} +\M{k+1}^{[e,\sigma]},
& \text{ if } \ell(\sigma)<\ell(\sigma 1),
\\
\M{k-1}^{[e,\sigma]} +\M{\red{k}}^{[e,\sigma]/\langle s_1\rangle},
& \text{ if } \ell(\sigma)>\ell(\sigma 1),
\qquad (k\ge 1).
\end{cases}
$
\item
$\TT{3k}^{\sigma} =
\begin{cases}
\M{3k}^{[e,\sigma]} +  \M{\blue{3k+1}}^{[e,\sigma]} +\M{3k+2}^{[e,\sigma]},
& \text{ if } \ell(\sigma)<\ell(\sigma 2),
\\
\M{3k}^{[e,\sigma]} +\M{\blue{3k+1}}^{[e,\sigma]/\langle s_2\rangle},
& \text{ if } \ell(\sigma)>\ell(\sigma 2),
\qquad (k\ge 1).
\end{cases}$
\end{enumerate}
\end{thm}

\begin{proof}
Here and throughout the paper, we always apply a translation functor $\E$ by tensoring with the adjoint $\g$-module and projecting to the block $\Bl_k$. Most of the time, we apply $\E$ to an initial tilting module $\TT{g}$ with $g=f_n^\sigma-2\delta$ to obtain the desired formula for the tilting module $\TT{n}^\sigma=\TT{f_n^\sigma}$. We shall call such an initial tilting module $\TT{g}$ with $g=f_n^\sigma-2\delta$ {\it standard}; otherwise it is called {\it nonstandard}.

The initial tilting modules used to establish Parts (1)--(3) are all standard except in Part~ (1) when $n=3k-3$. Note that the weight $f^\sigma_n-2\delta$ is atypical if and only if $n=3k-3$ and $k\ge 2$. For Part (1) in this particular case of $n=3k-3$ we use the following the nonstandard initial tilting modules to obtain the formulae listed in the theorem:
\begin{itemize}
\item
For $\sigma =e,1,121,212,1212$, we take $\TT{f_{3k-3}^{\sigma}-(\delta-\ep_3)}$;
\item
For $\sigma =2,12,21,2121,21212$, we take $T_{f^{\sigma}_{3k-3}-(\delta+\epsilon_2)}$;
\item
For $\sigma =12121,\wl$, we take $T_{f^{\sigma}_{3k-3}-(\delta+\epsilon_1)}$.
\end{itemize}

The initial tilting modules used to establish Part (4) for $\ell(\sigma)<\ell(\sigma1)$ and $k\ge 2$ are all standard.
Nonstandard initial tilting modules in Part (4) are used in the following cases: 
\begin{enumerate}
\item[(4-i)] Assume that $\ell(\sigma)<\ell(\sigma1)$ and $\underline{k=1}$.
\begin{itemize}
\item  For $\sigma=e,2,12$, we take $\TT{f_{k-1}^\sigma-(\delta-\ep_3)}$.
\item  For $\sigma=212$, we take $\TT{f_{k-1}^\sigma-(\delta+\ep_2)}$.
\item  For $\sigma=1212$, we take $\TT{f_{k-1}^\sigma-(\delta+\ep_1)}$.
\item  For $\sigma=21212$, we take $\TT{f_{k-1}^\sigma-(\delta-\ep_1)}$.
\end{itemize}
\item[(4-ii)] Assume that $\ell(\sigma)>\ell(\sigma1)$, and let $k\ge 1$ be arbitrary.
\begin{itemize}
\item  For $\sigma=1$, we take $\TT{f_{k-1}^\sigma-(\ep_2-\ep_3)}$.
\item  For $\sigma=21$, we take $\TT{f_{k-1}^\sigma-(\delta+\ep_1)}$.
\item  For $\sigma=121,\wl$, we take $\TT{f_{k-1}^\sigma-(\delta+\ep_2)}$.
\item  For $\sigma=2121,12121$, we take $\TT{f_{k-1}^\sigma-(\delta-\ep_3)}$.
\end{itemize}
\end{enumerate}

Nonstandard initial tilting modules in Part (5) of are only needed in the case $\ell(\sigma)>\ell(\sigma2)$. They are as follows:
\begin{itemize}
\item  For $\sigma=2,212$, we take $\TT{f_{3k}^\sigma-(\delta+\ep_1)}$.
\item  For $\sigma=12$, we take $\TT{f_{3k}^\sigma-(\ep_2-\ep_3)}$.
\item  For $\sigma=1212$, we take $\TT{f_{3k}^\sigma-(\delta+\ep_2)}$.
\item  For $\sigma=21212,\wl$, we take $\TT{f_{3k}^\sigma-(\delta-\ep_3)}$.
\end{itemize}

In this way, with each suitable choice of $g$ we obtain a module $\E \TT{g}$ whose Verma flags are given exactly by the RHS of the formuae stated in the theorem. (This step in reality is the most tedious and time-consuming part of this work, and it requires extensive use of Mathematica, as it takes several attempts to arrive at the suitable nonstandard initial tilting modules.)

\vspace{3mm}
By construction, in each case of (1)--(5) $M_{f_n^\sigma}$ appears in  $\E \TT{g}$ as a highest term. Therefore in order to prove that  $\E \TT{g}= \TT{f_n^\sigma}$, it remains to show that  $\E \TT{g}$ is indecomposable. Our main technical tool is Proposition~\ref{prop:flags}.

We observe that every formula in (1)--(5) has 2 or 3 different indices, to which we shall refer as 2 or 3 {\it layers} (of Verma flags). Noting that formulae with 2 layers in (1)--(5) are all multiplicity free, we conclude quickly from Proposition~\ref{prop:flags} that the resulting module $\E \M{g}$ is indecomposable and hence must be the tilting module $\TT{f}$.

So it remains to prove that $\E M_g$ with formulae of 3 layers, which occur in (4)--(5), are indecomposable. In these cases, we caution that multiplicity 2 does occur for $\sigma$ with $\ell(\sigma)\ge 2$, as 2 Verma modules of singular highest weights among $M_{\red{k}}^{[e,\sigma]}$ (respectively, $M_{\blue{3k+1}}^{[e,\sigma]}$) can be identified.
Let us first deal with Part (4), while Part (5) is parallel.

(4). We now work with $\TT{k-1}^\sigma$.
The case with $\sigma=e$ is clear: $\E \TT{g}$ with 3 Verma flags must be indecomposable by Proposition~\ref{prop:flags}, and hence is $\TT{k-1}^e$.

In case of $\sigma =s_2$, by Proposition~\ref{prop:flags}, the four terms in the first 2 layers are flags in $\TT{k-1}^{2}$; but then the remaining 2 terms cannot form a tilting module or a direct sum of tilting modules, and we are done.

For $\sigma=12$, we note
\begin{itemize}
\item $(f^{12}_{k-1},\delta-\ep_2)=0$,
\item $(f^{12}_{k-1}-\delta+\ep_2,\delta-\ep_3)=0$,
\item $f^{12}_{k-1}-(\delta-\ep_2)-(\delta-\ep_3) = f^{12}_{k+1}$,
\item ${\rm ht}(\delta-\ep_2)<{\rm ht}(\delta-\ep_3)$.
\end{itemize}
Thus, by Proposition~\ref{prop:flags}(5) we have $(T^{12}_{k-1}:M^{12}_{k+1})>0$. By Proposition~\ref{prop:flags}(6) we also have $(T^{12}_{k-1}:M^{s_2}_{k+1})>0$, $(T^{12}_{k-1}:M^{s_1}_{k+1})>0$, and $(T^{12}_{k-1}:M^{e}_{k+1})>0$. Since all layer 3 terms are part of the tilting module, we are done.

For $\sigma=212$, we observe that
\begin{itemize}
\item $(f^{212}_{k-1},\delta+\ep_3)=0$,
\item $(f^{212}_{k-1}-\delta\ep_3,\delta+\ep_2)=0$,
\item $f^{212}_{k-1}-(\delta+\ep_2)-(\delta+\ep_3) = f^{212}_{k+1}$,
\item ${\rm ht}(\delta+\ep_3)<{\rm ht}(\delta+\ep_2)$.
\end{itemize}
Using Proposition~\ref{prop:flags}(5) we have that $(T^{212}_{k-1}:M^{212}_{k+1})>0$, which implies that all layer 3 terms are part of the tilting module by Proposition~\ref{prop:flags}(6).

For $\sigma=1212$ we note
\begin{itemize}
\item $(f^{1212}_{k-1},\delta+\ep_3)=0$,
\item $(f^{1212}_{k-1}-(\delta+\ep_3),\delta+\ep_1)=0$,
\item $f^{1212}_{k-1}-(\delta+\ep_3)-(\delta+\ep_1) = f^{1212}_{2}$,
\item ${\rm ht}(\delta+\ep_3)<{\rm ht}(\delta+\ep_1)$.
\end{itemize}
The same argument as above by Proposition~\ref{prop:flags} shows that all layer 3 terms are part of the tilting module.

For $\sigma=21212$ we note
\begin{itemize}
\item $(f^{21212}_{k-1},\delta-\ep_2)=0$,
\item $(f^{21212}_{k-1}-(\delta-\ep_2),\delta-\ep_1)=0$,
\item $f^{21212}_{k-1}-(\delta-\ep_2)-(\delta-\ep_1) = f^{21212}_{2}$,
\item ${\rm ht}(\delta-\ep_2)<{\rm ht}(\delta-\ep_1)$.
\end{itemize}
The same argument as above using Proposition~\ref{prop:flags} shows that all layer 3 terms are part of the tilting module.

\vspace{2mm}

(5). We now work with $\TT{3k}^\sigma$.
The case with $\sigma=e$ is clear: $\E \TT{g}$ with 3 Verma flags must be indecomposable and hence is $\TT{3k}^e$.

In case of $\sigma =s_1$, by Proposition~\ref{prop:flags},  the first four terms in the first 2 layers are flags in $\TT{3k}^{1}$; but then the remaining 2 terms cannot form a tilting module or a direct sum of tilting modules, and we are done.

For $\sigma =21$, we note
\begin{itemize}
\item $(f^{21}_{3k},\delta-\ep_1)=0$,
\item $(f^{21}_{3k}-\delta+\ep_1,\delta+\ep_2)=0$,
\item $f^{21}_{3k}-(\delta-\ep_1)-(\delta+\ep_2) = f^{21}_{3k+2}$,
\item ${\rm ht}(\delta-\ep_1)<{\rm ht}(\delta+\ep_2)$.
\end{itemize}
Thus, by Proposition~\ref{prop:flags}(5) we have $(T^{21}_{3k}:M^{21}_{3k+2})>0$. By Proposition~\ref{prop:flags}(6) we also have $(T^{21}_{3k}:M^{s_2}_{3k+2})>0$, $(T^{21}_{3k}:M^{s_1}_{3k+2})>0$, $(T^{21}_{3k}:M^{e}_{3k+2})>0$. Since all layer 3 terms are part of the tilting module, we are done.

For $\sigma=121$ we observe that
\begin{itemize}
\item $(f^{121}_{3k},\delta-\ep_2)=0$,
\item $(f^{121}_{3k}-\delta+\ep_2,\delta+\ep_1)=0$,
\item $f^{121}_{3k}-(\delta-\ep_2)-(\delta+\ep_1) = f^{21}_{3k+2}$,
\item ${\rm ht}(\delta-\ep_2)<{\rm ht}(\delta+\ep_1)$.
\end{itemize}
Using Proposition~\ref{prop:flags}(5) we have that $(T^{121}_{3k}:M^{121}_{3k+2})>0$, which implies that all layer 3 terms are part of the tilting module by Proposition~\ref{prop:flags}(6).


For $\sigma=2121$ we note
\begin{itemize}
\item $(f^{2121}_{3k},\delta+\ep_3)=0$,
\item $(f^{2121}_{3k}-(\delta+\ep_3),\delta-\ep_1)=0$,
\item $f^{2121}_{3k}-(\delta+\ep_3)-(\delta-\ep_1) = f^{2121}_{2}$,
\item ${\rm ht}(\delta+\ep_3)<{\rm ht}(\delta-\ep_1)$.
\end{itemize}
The same argument as above by Proposition~\ref{prop:flags} shows that all layer 3 terms are part of the tilting module.

For $\sigma=12121$ we note
\begin{itemize}
\item $(f^{12121}_{3k},\delta+\ep_3)=0$,
\item $(f^{12121}_{3k}-(\delta+\ep_3),\delta-\ep_2)=0$,
\item $f^{12121}_{3k}-(\delta+\ep_3)-(\delta-\ep_2) = f^{12121}_{2}$,
\item ${\rm ht}(\delta+\ep_3)<{\rm ht}(\delta-\ep_2)$.
\end{itemize}
The same argument as above by Proposition~\ref{prop:flags} shows that all layer 3 terms are part of the tilting module.

The theorem is proved.
\end{proof}

We remark that the length of the Verma filtration of the tilting module in Part~(1) of Theorem~\ref{thm:k=1+sigma} is $4\ell(\sigma)$;  the lengths of the Verma filtrations in the tilting modules in Parts (2)-(3) are $3\ell(\sigma)$; the lengths of Verma filtrations in Part~(4) are $6\ell(\sigma)$ and $3\ell(\sigma)$, respectively;  same for Part~ (5).

\subsection{Formulae for $\TT{\red{0}}^{\sigma}$ in the block $\Bl_0$}

Theorem~\ref{thm:k=1+sigma} covers all blocks $\Bl_k$ with $k\ge 1$ and part of block $\Bl_0$.
The only missing case not covered by Theorem~\ref{thm:k=1+sigma}, $\TT{\red{0}}^{\sigma}$ in the block $\Bl_0$, will be treated in the following.

\begin{thm}  [Block $\Bl_0$]
   \label{thm:k=0:sigma}
The following formulae hold for tilting modules in the block $\Bl_0$: for $\sigma \in W_2$ such that $\ell(\sigma) >\ell(\sigma 1)$,
\begin{align}
\TT{\red{0}}^{\sigma 1} &=\TT{\red{0}}^{\sigma }
=  \M{\red{0}}^{[e,\sigma]/\langle s_1\rangle} +\M{\blue{1}}^{[e,\sigma ]} +\M{2}^{[e,\sigma ]} \text{ for }\sigma\not= \wl,
 \\
\TT{\red{0}}^{21212}&=\TT{\red{0}}^{\wl}
=  \M{\red{0}}^{[e,\wl]/\langle s_1\rangle} +\M{\blue{1}}^{[e,\wl]/\langle s_2\rangle}.
\label{k=0:0wl}
\end{align}
\end{thm}

\begin{proof}
We first prove the formula \eqref{k=0:0wl} for $\TT{\red{0}}^{21212}$, which requires a nonstandard initial tilting module.
By applying the translation functor $\E$ on the nonstandard initial tilting module of highest weight $g=f^{21212}_0-(\delta-\ep_3)$, which corresponds to the symbol $[-3/2|0,0,0]$, we obtain the module $\E \TT{g}$ whose Verma flags are given on the RHS of the formula for $\TT{\red{0}}^{21212}$ in the theorem.
Since $\E \TT{g}$ is multiplicity free with 2 layers and hence indecomposable by Proposition~ \ref{prop:flags},  it must be $\TT{\red{0}}^{21212}$.

The explicit formulae for $\TT{\red{0}}^{\sigma}$ with $\sigma \neq 21212$ are obtained by applying a suitable translation functor $\E$ to the standard initial tilting modules. It remains to show that $\E \TT{g}$ is indecomposable, and we provide the details case-by-case below.

For $\sigma=e$, by Proposition~\ref{prop:flags} both $\M{\blue 1}^{s_1}$ and $\M{\blue 1}^e$ appear in $\TT{\red 0}^{e}$, while the last two terms cannot form a tilting module or a direct sum of tilting modules.

For $\sigma=s_2$, We note
\begin{itemize}
\item $(f^{s_2}_0,\delta-\ep_1)=0$,
\item $(f^{s_2}_0-\delta+\ep_1,\delta+\ep_2)=0$,
\item $f^{s_2}_0-(\delta-\ep_1)-(\delta+\ep_2) = f^{21}_2$,
\item ${\rm ht}(\delta-\ep_1)<{\rm ht}(\delta+\ep_2)$.
\end{itemize}
Thus, by Proposition~\ref{prop:flags}(5) we have $(T^{s_2}_{\red{0}}:M^{21}_2)>0$. By Proposition~\ref{prop:flags}(6) we also have $(T^{s_2}_{\red{0}}:M^{s_2}_2)>0$, $(T^{s_2}_{\red{0}}:M^{s_1}_2)>0$, $(T^{s_2}_{\red{0}}:M^{e}_2)>0$. Since all layer 3 terms are part of the tilting module, we are done.

For $\sigma=12$, we note
\begin{itemize}
\item $(f^{12}_{0},\delta-\ep_2)=0$,
\item $(f^{12}_{0}-(\delta-\ep_2),\delta+\ep_1)=0$,
\item $f^{12}_{0}-(\delta-\ep_2)-(\delta+\ep_1) = f^{121}_{2}$,
\item ${\rm ht}(\delta-\ep_2)<{\rm ht}(\delta+\ep_1)$.
\end{itemize}
Thus, by Proposition~\ref{prop:flags}(5) we have $(T^{12}_{\red 0}:M^{121}_{2})>0$. By Proposition~\ref{prop:flags}(6) all layer 3 terms are part of the tilting module.

For $\sigma=212$, we note
\begin{itemize}
\item $(f^{212}_{0},\delta+\ep_3)=0$,
\item $(f^{212}_{0}-(\delta+\ep_3),\delta-\ep_1)=0$,
\item $f^{212}_{0}-(\delta+\ep_3)-(\delta-\ep_1) = f^{2121}_{2}$,
\item ${\rm ht}(\delta+\ep_3)<{\rm ht}(\delta-\ep_1)$.
\end{itemize}
Thus, by Proposition~\ref{prop:flags}(5) we have $(T^{212}_{\red 0}:M^{2121}_{2})>0$. By Proposition~\ref{prop:flags}(6) all layer 3 terms are part of the tilting module.

For $\sigma=1212$, we note
\begin{itemize}
\item $(f^{1212}_{0},\delta+\ep_3)=0$,
\item $(f^{1212}_{0}-(\delta+\ep_3),\delta-\ep_2)=0$,
\item $f^{1212}_{0}-(\delta+\ep_3)-(\delta-\ep_2) = f^{12121}_{2}$,
\item ${\rm ht}(\delta+\ep_3)<{\rm ht}(\delta-\ep_2)$.
\end{itemize}
Thus, by Proposition~ \ref{prop:flags}(5) we have $(T^{1212}_{\red{0}}:M^{12121}_{2})>0$. By Proposition~ \ref{prop:flags}(6) all layer 3 terms are part of $T^{1212}_{\red{0}}$, and hence the formula for $\sigma=1212$ follows.

The theorem is proved.
\end{proof}
We note that the length of the Verma filtration for the tilting module $\TT{\red{0}}^{\sigma}$ (for $\sigma \neq \wl$) in Theorem~\ref{thm:k=0:sigma} is $5\ell(\sigma)$, while the length of the Verma filtration in the tilting module $\TT{\red{0}}^{\wl}$ in the formula \eqref{k=0:0wl}   is $2\ell(\wl)=12$.

\begin{rem}
 \label{rem:various}
\begin{enumerate}
\item
In Theorem~\ref{thm:k=1+sigma}(4)-(5) and Theorem~\ref{thm:k=0:sigma}, the $w$ appearing in $\M{\red{k}}^w$ or $\M{\blue{3k+1}}^w$ (of singular highest weights) may not be of minimal length in $W_2/\langle s_i \rangle$. A multiplicity 2 could occur after such terms are rewritten uniformly via minimal length elements in $W_2/\langle s_i \rangle$. For example, Theorem~\ref{thm:k=1+sigma}(4) with $\sigma =12$ reads
\[
\TT{k-1}^{12} =
\M{k-1}^{[e,12]} + 2 \M{\red{k}}^{e} +  \M{\red{k}}^{\{2,12\}} +\M{k+1}^{[e,12]}.
\]



\item
Note $f^\sigma_n -2\delta$ is atypical if and only $n=3k-3$, and we have used non-standard initial tilting modules to obtain $T^{\sigma}_{3k-3}$ for $k\ge 2$. In spite of this the formulae for $\TT{3k-3}^\sigma$ fit into the generic formulae given in Theorem~\ref{thm:k=1+sigma}(1).
\end{enumerate}
\end{rem}

\section{Character formulae for tilting modules in $\OO$, II}
  \label{sec:charT0}

In this section, we provide formulae for Verma flags of tilting modules $T^{0\sigma}_n$, for $\sigma \in W_2$, in all blocks $\Bl_k$ with the exception of one particular tilting module in $\Bl_0$.

\subsection{Formulae for $\TT{n}^{0\sigma}$ in the blocks $\Bl_k$}

In the formulae below, recalling the Weyl group $W_2$ of $G_2$, we denote by $A_1$ the Weyl group of $\sll$ and so $W=A_1\times W_2$; moreover, the superscript $0$ denotes $s_0 \in A_1$. Recall the shorthand notation $\M{n}^D$, for $D\subseteq W$, from \eqref{MnD}.

\begin{thm}
  \label{thm:0:k=2+sigma}
The following formulae hold for tilting modules in the block $\Bl_k$: for $\sigma\in W_2$,
\begin{enumerate}
  \item
$\TT{n}^{0\sigma} =  \M{n}^{A_1\times [e,\sigma]}  +\M{n-1}^{A_1\times [e,\sigma]},
\quad \forall n \in \N \backslash \{0,\red{k}, k+1, \blue{3k+1}, 3k+2\}, \qquad (k\ge 0).
$
\item
$ \TT{0}^{0\sigma} =
\begin{cases}
\M{0}^{A_1\times [e, \sigma]}+\M{0}^{[e,\sigma]2} +\M{1}^{[e, \sigma]},
& \text{ if } \ell(\sigma)<\ell(\sigma 2),
\\
\M{0}^{A_1\times [e, \sigma]},
& \text{ if } \ell(\sigma)>\ell(\sigma 2)
\qquad (k\ge 2).
\end{cases}
$
\item
$ \TT{\red{k}}^{0 \sigma 1} =\TT{\red{k}}^{0\sigma}
=  \M{\red{k}}^{A_1\times [e,\sigma]/\langle s_1\rangle} +\M{k-1}^{A_1\times [e,\sigma ]}$,
\quad if $\ell(\sigma) >\ell(\sigma 1), \qquad (k\ge 1)$.
\item
$ \TT{\blue{3k+1}}^{0\sigma 2} =\TT{\blue{3k+1}}^{0\sigma}
=  \M{\blue{3k+1}}^{A_1\times [e,\sigma]/\langle s_2\rangle} +\M{3k}^{A_1\times [e,\sigma ]}$,
if $\ell(\sigma) >\ell(\sigma 2), \qquad (k\ge 1). $
\item
$ \TT{k+1}^{0\sigma} =
\begin{cases}
\M{k+1}^{A_1\times [e,\sigma]} +\M{\red{k}}^{A_1\times [e,\sigma]} +\M{{k-1}}^{A_1\times [e,\sigma]},
& \text{ if } \ell(\sigma)<\ell(\sigma 1),
\\
\M{k+1}^{A_1\times [e,\sigma]} +\M{\red{k}}^{A_1\times [e,\sigma]/\langle s_1\rangle},
& \text{ if } \ell(\sigma)>\ell(\sigma 1),
\qquad (k\ge 1).
\end{cases}
$
\item
$ \TT{3k+2}^{0\sigma} =
\begin{cases}
\M{3k+2}^{A_1\times [e,\sigma]} +\M{\blue{3k+1}}^{A_1\times [e,\sigma]} +\M{3k}^{A_1\times [e,\sigma]},
& \text{ if } \ell(\sigma)<\ell(\sigma 2),
\\
\M{3k+2}^{A_1\times [e,\sigma]} +\M{\blue{3k+1}}^{A_1\times [e,\sigma]/\langle s_2\rangle},
& \text{ if } \ell(\sigma)>\ell(\sigma 2),
\qquad (k\ge 1).
\end{cases}
$
\end{enumerate}
\end{thm}


\begin{proof}
(1). Part (1) follows immediately from Lemma \ref{jantzen:sum} once we show that a translation functor applied to the standard initial tilting module (i.e., with highest weight shifted down by $2\delta$) produces the formula on the right hand side except for the case $n\not=3k+5$. Note that the weight $f^{0\sigma}_n -2\delta$ is atypical if and only $n=3k+5$ or $``k=0,n=2"$. The tilting modules $\TT{n}^{0\sigma}$ for $n=3k+5$ are obtained by translating from the following nonstandard initial tilting modules:
\begin{itemize}
\item For $\sigma=e,1,12,21212,121212$, we take $T_{f^{0\sigma}_{3k+5}-(\delta+\ep_2)}$;
\item For $\sigma=2,21,121,212,1212,2121,12121$, we take $T_{f^{0\sigma}_{3k+5}-(\delta-\ep_3)}$.
\end{itemize}

(2). To prove the first identity in Part (2), we apply the translation functor to the standard initial tilting module. We note that we can first subtract the weight $f_0^{0\sigma}$ by $\delta$ and then we can find an odd positive root $\beta$ such that $\text{ht}\beta>\text{ht}\delta$ with $\langle f_0^\sigma,\beta\rangle=0$ and $f_0^{\sigma}-\beta=f_1^{\sigma}$. This shows that $M^{[e,\sigma]}_1$ is a part of $\TT{0}^{0\sigma}$. From this and Proposition~\ref{prop:flags} we deduce that the tilting module has the claimed Verma flag structure.

The second formula of (2) is obtained  by applying translation functors  to the following nonstandard initial tilting modules:
\begin{itemize}
\item  For $\sigma=212,1212$, we take $\TT{f_{0}^{0\sigma} -(\delta+\ep_2)}$;
\item  For $\sigma=2,12,21212,121212$, we take $\TT{f_{k+1}^{0\sigma} -(\delta-\ep_3)}$.
\end{itemize}

(3). The formulae in  (3) are obtained  by applying the translation to standard initial tilting modules. Indecomposability is straightforward using Proposition~\ref{prop:flags} as the formulae are multiplicity free and of 2 layers only.

(4). The formulae in  (4) are obtained  by applying the translation to standard initial tilting modules. Indecomposability is straightforward using Proposition~\ref{prop:flags} as the formulae are multiplicity free and of 2 layers only.

(5).
The initial tilting modules used to establish Part (5) for $\ell(\sigma)<\ell(\sigma1)$ and $k\ge 2$ are all standard.
Nonstandard initial tilting modules in Part (5) are used as follows:
\begin{enumerate}
\item[(5-i)] Assume that $\ell(\sigma)<\ell(\sigma1)$ and $\underline{k=1}$.
\begin{itemize}
\item  For $\sigma=e,2$, we take $\TT{f_{k+1}^{0\sigma}-\delta}$;
\item  For $\sigma=12, 21212$, we take $\TT{f_{k+1}^{0\sigma} -(\delta+\ep_2)}$;
\item  For $\sigma=212, 1212$, we take $\TT{f_{k+1}^{0\sigma} -(\delta-\ep_3)}$.
\end{itemize}
\item[(5-ii)] Assume that $\ell(\sigma)>\ell(\sigma1)$, and let $k\ge 1$ be arbitrary.
\begin{itemize}
\item  For $\sigma=1,2121$, we take $\TT{f_{k+1}^{0\sigma}-(\delta+\ep_2)}$;
\item  For $\sigma=21,121$, we take $\TT{f_{k+1}^{0\sigma} -(\delta-\ep_3)}$;
\item  For $\sigma=12121$, we take $\TT{f_{k+1}^{0\sigma} -(\delta+\ep_1)}$;
\item  For $\sigma=\wl$ we take $\TT{f_{k+1}^{0\sigma} -(\delta-\ep_1)}$.
\end{itemize}
\end{enumerate}
Indecomposability for the tilting modules in the second formula of (5) with two layers is clear by Proposition~\ref{prop:flags}.

For the tilting modules in the first formula of (5), one shows that for these weights $f^{o\sigma}_{k+1}$ we have $(f^{0\sigma}_{k+1},\alpha)=0$ and $(f^{0\sigma}_{k+1}-\alpha,\beta)$ for some $\alpha,\beta\in\Phi^+_{\bar 1}$ with $\text{ht}\alpha<\text{ht}\beta$ and $f^{0\sigma}_{k+1}-\alpha-\beta=f^{0\sigma}_{k-1}$. This implies immediately that all the components on the right hand side of the formula in the third layer must be part of the tilting module by Proposition~\ref{prop:flags}, which in turn implies immediately the indecomposability for all these tilting modules.

(6). The first formula in (6) is obtained by applying translations to the standard initial tilting modules. The second formula in (6) is obtained by applying the translations to the following nonstandard initial tilting modules:
\begin{itemize}
\item  For $\sigma=2,12$, we take $\TT{f_{3k+2}^{0\sigma}-(\delta-\ep_3)}$;
\item  For $\sigma=212$, we take $\TT{f_{3k+2}^{0\sigma} -(\delta+\ep_2)}$;
\item  For $\sigma=1212$, we take $\TT{f_{3k+2}^{0\sigma} -(\delta+\ep_1)}$;
\item  For $\sigma=21212$, we take $\TT{f_{3k+2}^{0\sigma} -(\delta-\ep_1)}$;
\item  For $\sigma=\wl$, we take $\TT{f_{3k+2}^{0\sigma} -(\delta-\ep_2)}$.
\end{itemize}
Indecomposability follows using Proposition~\ref{prop:flags} in a standard fashion as for (5).

The proof of Theorem~\ref{thm:0:k=2+sigma} is completed.
\end{proof}

We remark that the length of the Verma filtration in Part~(1) of Theorem~\ref{thm:0:k=2+sigma} is $8\ell(\sigma)$;
the lengths of Verma filtrations in Part~(2) are $8\ell(\sigma)$ and $4\ell(\sigma)$, respectively; the lengths of Verma filtrations in the tilting modules in Part (3)-(4) are $6\ell(\sigma)$.
the lengths of Verma filtrations in Part~(5) are $12\ell(\sigma)$ and $6\ell(\sigma)$, respectively;  the Verma lengths in Part~ (6) are identical to (5).

Note that Theorem~\ref{thm:0:k=2+sigma} covers all blocks $\Bl_k$ with $k\ge 2$.
The only cases not covered by  Theorem~\ref{thm:0:k=2+sigma} are: $\TT{\red{0}}^{0\sigma}$, $\TT{\blue{1}}^{0\sigma}$ , $\TT{2}^{0\sigma}$
in $\Bl_0$ and  $\TT{0}^{0\sigma}$ 
in $\Bl_1$. These exceptional cases will be treated in  the following two subsections. 

\subsection{Formulae for $\TT{0}^{0\sigma}$ in the block $\Bl_1$}

When combined with Theorem~\ref{thm:0:k=2+sigma}, the following theorem completes the character formulae for tilting modules $\TT{n}^{0\sigma}$ in the block $\Bl_1$ for all $n\in \N$ and $\sigma \in W_2$.

\begin{thm}  [Block $\Bl_1$]
   \label{thm:0:k=1:sigma}
The following formulae hold for tilting modules in the block $\Bl_1$: for $\sigma \in W_2$,
\[
\TT{0}^{0\sigma} =
\begin{cases}
\M{0}^{A_1}+ \M{0}^{2}+ \M{\red{1}}^e+\M{2}^e,
& \quad  \text{ if }  \sigma =e,
\\
\M{0}^{A_1\times[e,\sigma]} +\M{0}^{[e,\sigma]2} +\M{{\red{1}}}^{[e,\sigma]/\langle s_1\rangle},
& \quad  \text{ if } \ell(\sigma) <\ell(\sigma 2), \; \sigma\not=e,
\\
\M{0}^{A_1\times [e,\sigma]},
& \quad \text{ if } \ell(\sigma) >\ell(\sigma 2).
\end{cases}
\]
\end{thm}

\begin{proof}
The formula for $\sigma=e$ is obtained by applying a translation functor to the standard initial tilting module.

Then we apply translation functors to the following nonstandard initial tilting modules when $\ell(\sigma) <\ell(\sigma 2)$ with $\sigma \neq e$:
\begin{itemize}
\item  For $\sigma=1$, we take $\TT{f_{0}^{0\sigma}-(\delta-\ep_1)}$;
\item  For $\sigma=21$, we take $\TT{f_{0}^{0\sigma} -(\delta+\ep_1)}$;
\item  For $\sigma=121,2121,12121$, we take $\TT{f_{0}^{0\sigma} -(\delta-\ep_3)}$.
\end{itemize}
To prove the indecomposability we note that we can find an odd positive root $\beta$ such that $\text{ht}\beta>\text{ht}\delta$, $\langle f_0^\sigma,\beta\rangle=0$ and $f_0^{\sigma}-\beta=f_1^{\sigma}$.

Finally we apply translation functors to the following nonstandard initial tilting modules when $\ell(\sigma) >\ell(\sigma 2)$:
\begin{itemize}
\item  For $\sigma=212,1212$, we take $\TT{f_{0}^{0\sigma} -(\delta+\ep_2)}$;
\item  For $\sigma=2,12,21212,121212$, we take $\TT{f_{0}^{0\sigma} -(\delta-\ep_3)}$.
\end{itemize}
Indecomposibility of the resulting modules follows readily from applying Proposition~\ref{prop:flags}.


The theorem is proved.
\end{proof}

 We note that the lengths of Verma filtrations for the tilting modules $\TT{0}^{0\sigma}$ in Theorem~\ref{thm:0:k=1:sigma} are $5,  7\ell(\sigma)$, and $4\ell(\sigma)$, respectively.

\subsection{Formulae for $\TT{\red{0},\blue{1},2}^{0\sigma}$ in the block $\Bl_0$}

When combined with Theorem~\ref{thm:0:k=2+sigma}, the following theorem completes the character formulae for tilting modules $\TT{n}^{0\sigma}$ in the block $\Bl_0$ for all $n\in \N$ and $\sigma \in W_2$, except $\TT{2}^0$.

\begin{thm} [Block $\Bl_0$]
   \label{thm:0:k=0:sigma}
The following formulae hold for tilting modules in the block $\Bl_0$, for $\sigma \in W_2$.

\begin{enumerate}
\item
Let $\sigma =2, 12, 212$. Suppose that $\ell(i\sigma)<\ell(\sigma)$, for $i\in\{1,2\}$. Then
\begin{align*}
\TT{\red{0}}^{0\sigma 1} &= \TT{{\red{0}}}^{0\sigma}
 = \M{\red{0}}^{A_1 \times [e,\sigma]/\langle s_1\rangle} + \M{\red{0}}^{[e,\sigma]\cup\{i\sigma12,\sigma12\}} + \M{\blue{1}}^{[e,\sigma1]}.
 \end{align*}
 Moreover, we have
 \begin{align*}
\TT{\red{0}}^{01} &= \TT{\red{0}}^{0} = \M{\red{0}}^{A_1} + \M{\red{0}}^{\{2,12\}} + \M{\blue{1}}^{[e,1]},
\\
 \TT{\red{0}}^{01212 1} &= \TT{\red{0}}^{01212}
 = \M{\red{0}}^{A_1 \times [e,1212]/\langle s_1\rangle} + \M{\red{0}}^{[e,12]/\langle s_1\rangle\cup\{21212\}} + \M{\blue{1}}^{[e,121]/\langle s_2\rangle},
 \\
\TT{\red{0}}^{0w_0} &=\TT{\red{0}}^{21212}=\M{\red{0}}^{A_1\times W_2/\langle s_1\rangle}.
 \end{align*}

\item
$\TT{\blue{1}}^{0\sigma 2}
= \TT{\blue{1}}^{0\sigma} =
\begin{cases}
\M{\blue{1}}^{A_1\times [e,\sigma]/\langle s_2\rangle} +\M{\red{0}}^{A_1\times [e,\sigma ]},
& \text{ if }  \ell(\sigma)>\ell(\sigma 2), \; \sigma \not=\wl,
\\
\M{\blue{1}}^{A_1\times [e,w_0]/\langle s_2\rangle} + \M{\red{0}}^{A_1\times [e,w_0]/\langle s_1\rangle},
& \text{ if } \sigma=\wl.
\end{cases}
$
\item
$\TT{2}^{0\sigma} =
\begin{cases}
\M{2}^{A_1\times [e,\sigma]} +\M{\blue{1}}^{A_1\times [e,\sigma]} +\M{{\red{0}}}^{A_1\times {[e,\sigma]/\langle s_1\rangle}},
&   \text{ if } \ell(\sigma) <\ell(\sigma 2), \; \sigma\not=e,
\\
\M{2}^{A_1\times [e,\sigma]} +\M{\blue{1}}^{A_1\times {[e,\sigma]/\langle s_2\rangle}},
&  \text{ if } \ell(\sigma) >\ell(\sigma 2).
\end{cases}
$
\end{enumerate}
\end{thm}

\begin{proof}
(1).  The tilting module $\TT{0}^{0212121}$ is obtained by translating from the initial tilting module $\TT{g}$ with $g=f^{0w_0}_0-(\delta-\ep_3)=[-1/2|0,0,0]$. All other tilting modules are obtained by applying translation functors to standard initial tilting modules.

In case $\sigma\not=212,1212$, Part (1) can be established using Proposition~\ref{prop:flags} and the formulas in Theorem \ref{thm:k=0:sigma}, since in these cases the tilting modules $\TT{0}^\tau$ have components of the form $\M{2}^\kappa$ in the case $\tau\in W_2$ and $\tau\not=21212,w_0$.

If $\sigma=212$, then there is a component of the form $\M{0}^{21212}$. But the tilting module $\TT{0}^{w_0}$ contains a component $\M{1}^{w_0}$, which is not a component of the formula above. This proves the case for $\sigma=212$.

Now we consider the case of $\sigma=1212$. Applying the usual translation functor to the tilting module of highest weight $f^{01212}_0-2\delta$ we get the formula
\begin{align}\label{big:n=k=0:01212}
\M{0}^{A_1 \times [e,1212]/\langle s_1\rangle} + \M{0}^{[e,1212]\cup\{21212,121212\}} + \M{1}^{[e,12121]}.
\end{align}

Consider the sequence of weights $f_0^{0\sigma}=[1/2|-1/2,1,-1/2]$, $f_0^\sigma=f_0^{0\sigma}-[1|0,0,0]=[-1/2|-1/2,1,-1/2]$, $f_0^\sigma-(\delta+\ep_1)=f_1^{121}=[-3/2|-3/2,3/2,0]$. Since $\delta<\delta+\ep_1$ we see that $\M{1}^{121}$ is a component of $\TT{0}^{0\sigma}$. This proves that $\M{1}^{\{e,1,21,121\}}$ are components of $\TT{0}^{0\sigma}$.

We shall now argue that \eqref{big:n=k=0:01212} above is the sum of the tilting modules $\TT{0}^{01212}\oplus\TT{0}^{w_0}$. First, note that the weight space $(\M{0}^0)_{f_0^{1212}}$ is $3$-dimensional. The composition factor $L^{212}_0$ of $M^0_0$ has such a non-zero weight space and the composition factor $L^{1212}_0$ has clearly such a weight space as well. Thus, \eqref{big:n=k=0:01212} is isomorphic to $\TT{0}^{01212}\oplus\TT{0}^{w_0}$ is equivalent to $[\M{0}^{0}:L_{0}^{1212}]=1$. Now, if $L_0^{1212}$ appears with composition multiplicity $2$, then this would mean that $L^0_0$ does not have such a weight space. But if $v$ is a highest weight vector of $M^0_0$, then the vector $f_{\alpha_1}f_{\alpha_2}f_{\alpha_3}v$ is a vector of that weight in $M^0_0$. Now, since we have $e_{\alpha_i}f_{\alpha_i}v_\la$ is a non-zero scalar multiple of $v$, we see that $e_{\alpha_1}e_{\alpha_2}e_{\alpha_3}f_{\alpha_1}f_{\alpha_2}f_{\alpha_3}v$ is a non-zero multiple of $v$. This proves that $(L^0_0)_{f^{1212}_0}\not=0$ and hence $[\M{0}^{0}:L_{0}^{1212}]=1$, and whence (1) for $\sigma=1212$.

\vspace{3mm}
(2). 
The formulae are obtained applying translation functors to the standard initial tilting modules, except for $\TT{1}^{0212121}$, which is obtained by translating from the initial tilting module $\TT{g}$, where $g =f^{0w_0}_1-(\ep_2-\ep_3)=[3/2|0,0,0]$. Indecomposability is straightforward using Proposition~\ref{prop:flags} and Part~(1).

\vspace{3mm}

(3). The formulae for the tilting modules $\TT{2}^{0\sigma}$ are obtained by applying translation functors  to the following nonstandard initial tilting modules:
\begin{itemize}
\item
For $\sigma =1,21212$, take $\TT{f^{0\sigma}_2-(\delta-\ep_1)}$;
\item
For $\sigma =12$, take $\TT{f^0_2-\delta}$;
\item
For $\sigma =21,1212$, take  take $\TT{f^{0\sigma}-(\delta+\ep_1)}$;
\item
For $\sigma =212$, take  $\TT{f^{0\sigma}_2-(\delta+\ep_2)}$;
\item
For $\sigma =121,\wl$, take $\TT{f^{0\sigma}_2-(\delta-\ep_2)}$;
\item
For $\sigma =2,2121,12121$, take $\TT{f^{0\sigma}_2-(\delta-\ep_3)}$.
\end{itemize}
Indecomposability of all the formulae with $\sigma\neq e$ follows directly from Proposition~\ref{prop:flags}.

\vspace{2mm}
The proof of Theorem~\ref{thm:0:k=0:sigma} is completed.
\end{proof}

We remark that the length of the Verma filtration of the tilting module $\TT{\blue{1}}^{0\sigma}$ in Theorem~\ref{thm:0:k=0:sigma} with $\sigma \neq \wl$ is $6 \ell(\sigma)$; the lengths of Verma filtrations of the tilting modules $\TT{2}^{0\sigma}$ in Theorem~\ref{thm:0:k=0:sigma} are $10\ell(\sigma)$ and $6 \ell(\sigma)$, respectively.

 \begin{rem}
  \label{rem:T20}
Among all tilting modules in $\OO$, the tilting module $\TT{2}^0$ in $\Bl_0$ is the only one whose Verma filtration remains to be determined. We can show that $\TT{2}^0$ must be a summand of a module with a Verma flag $\M{2}^{0,1,2,2,e,e,e} +\M{1}^{0,1,e,e,e} +\M{0}^{0,2,e} +\M{3}^{e}$ (note this module has 4 layers); we can show that $\TT{2}^0$ is one of 5 explicit possible submodules (of Verma lengths $10, 11, 13, 14, 16$, respectively), but we are unable to pin it down.
\end{rem}

\begin{rem}
A Verma module with multiplicity 3 appears in the Verma filtrations of the tilting modules for $\sigma=12, 212$ in Theorem~\ref{thm:0:k=0:sigma}(1). For example, in terms of minimal coset representatives we can write
\begin{align*}
\TT{\red{0}}^{012}
 = \M{\red{0}}^{\{0,02,012, 212, 1212\}} + 3 \M{\red{0}}^{e} + 2 \M{\red{0}}^{\{2,12\}}  + 2 \M{\blue{1}}^{e} + 2\M{\blue{1}}^{1} + \M{\blue{1}}^{\{21,121\}}.
 \end{align*}
These 2 tilting modules (and possibly $\TT{2}^0$ too, see Remark~\ref{rem:T20}) are the only cases where multiplicity 3 appears in their Verma filtrations among all tilting modules in $\OO$. The Verma multiplicity is at most 2 in all other cases of tilting modules in $\OO$.
 \end{rem}

\section{Character formulae for projective modules in $\OO$}
  \label{sec:proj}

In this section, we describe the Verma flags for projective covers in the BGG category $\OO$, by using the formulae for Verma flags of tilting modules in $\OO$ obtained in the previous sections.

\subsection{Formulae for projective modules in $\Bl_k$}

Applying Theorem~\ref{thm:k=1+sigma} and Soergel duality \eqref{tiltingD}, we obtain the following character formulae for the projective modules  $\PP{n}^{0\tau}$, for $\tau\in W_2$, in the block $\Bl_0$. The notations below are converted from formulae in Theorem~\ref{thm:k=1+sigma} by setting $\tau =\wl \sigma$. Note that $s_0\wl \la = -\la$, for $\la \in X$.

\begin{thm}
  \label{thm:k=1+sigma:P}
The following formulae hold for projective modules in the block $\Bl_k$: for $\tau \in W_2$,
\begin{enumerate}
\item
$\PP{n}^{0\tau} =  \M{n}^{0[\tau,\wl]}  +\M{n+1}^{0[\tau,\wl]},\quad \forall n \in \N \backslash \{k-1, k, 3k, 3k+1\},\qquad (k\ge 0).
$
\item
$\PP{\blue{3k+1}}^{0\tau 2} =\PP{\blue{3k+1}}^{0\tau}
=  \M{\blue{3k+1}}^{0[\tau,\wl]/\langle s_2\rangle} +\M{3k+2}^{0[\tau ,\wl]}, \quad
\text{ if } \ell(\tau) < \ell(\tau 2),
\qquad (k\ge 0).
$
\item
$\PP{\red{k}}^{0\tau 1} =\PP{\red{k}}^{0\tau}
=  \M{\red{k}}^{0[\tau,\wl]/\langle s_1\rangle} +\M{k+1}^{0[\tau , \wl]}, \qquad\;\;\,
\text{ if } \ell(\tau)< \ell(\tau 1), \qquad (k\ge 1).$
\item
$\PP{k-1}^{0\tau} =
\begin{cases}
\M{k-1}^{0[\tau,\wl]} +  \M{\red{k}}^{0[\tau,\wl]} +\M{k+1}^{0[\tau,\wl]},
& \text{ if } \ell(\tau)>\ell(\tau 1),
\\
\M{k-1}^{0[\tau,\wl]} +\M{\red{k}}^{0[\tau,\wl]/\langle s_1\rangle},
& \text{ if } \ell(\tau)<\ell(\tau 1),
\qquad (k\ge 1).
\end{cases}
$
\item
$ \PP{3k}^{0\tau} =
\begin{cases}
\M{3k}^{0[\tau,\wl]} +  \M{\blue{3k+1}}^{0[\tau,\wl]} +\M{3k+2}^{0[\tau,\wl]},
& \text{ if } \ell(\tau)>\ell(\tau 2),
\\
\M{3k}^{0[\tau,\wl]} +\M{\blue{3k+1}}^{0[\tau,\wl]/\langle s_2\rangle},
& \text{ if } \ell(\tau)<\ell(\tau 2),
\qquad (k\ge 1).
\end{cases}
$
\end{enumerate}
\end{thm}

Applying Theorem~\ref{thm:0:k=2+sigma} and Soergel duality \eqref{tiltingD}, we obtain the following character formulae for the projective modules $\PP{n}^{\tau}$, for $\tau\in W_2$, in the block $\Bl_k$.

\begin{thm}
  \label{thm:0:k=2+sigma:P}
The following formulae hold for projective modules in the block $\Bl_k$: for $\tau \in W_2$,
\begin{enumerate}
  \item
$\PP{n}^{\tau} =  \M{n}^{A_1\times [\tau,\wl]}  +\M{n-1}^{A_1\times [\tau,\wl]},
\quad \forall n \in \N \backslash \{0,\red{k}, k+1, \blue{3k+1}, 3k+2\}, \quad (k\ge 0).
$
\item
$ \PP{0}^{\tau} =
\begin{cases}
\M{0}^{A_1\times [\tau,\wl]}+\M{0}^{0[\tau,\wl]2} +\M{1}^{0[\tau,\wl]},
& \text{ if } \ell(\tau)>\ell(\tau 2),
\\
\M{0}^{A_1\times [\tau,\wl]},
& \text{ if } \ell(\tau)<\ell(\tau 2),
\qquad (k\ge 2).
\end{cases}
$
\item
$ \PP{\red{k}}^{\tau 1} =\PP{\red{k}}^{\tau}
=  \M{\red{k}}^{A_1\times [\tau,\wl]/\langle s_1\rangle} +\M{k-1}^{A_1\times [\tau,\wl]}$,
\quad if $\ell(\tau) < \ell(\tau 1), \qquad (k\ge 1).
$
\item
$ \PP{\blue{3k+1}}^{\tau 2} =\PP{\blue{3k+1}}^{\tau}
=  \M{\blue{3k+1}}^{A_1\times [\tau,\wl]/\langle s_2\rangle} +\M{3k}^{A_1\times [\tau,\wl]}$,
if $\ell(\tau) <\ell(\tau 2), \qquad (k\ge 1).
$
\item
$ \PP{k+1}^{\tau} =
\begin{cases}
\M{k+1}^{A_1\times [\tau,\wl]} +\M{\red{k}}^{A_1\times [\tau,\wl]} +\M{{k-1}}^{A_1\times [\tau,\wl]},
& \text{ if } \ell(\tau)>\ell(\tau 1),
\\
\M{k+1}^{A_1\times [\tau,\wl]} +\M{\red{k}}^{A_1\times [\tau,\wl]/\langle s_1\rangle},
& \text{ if } \ell(\tau)<\ell(\tau 1),
\qquad (k\ge 1).
\end{cases}
$
\item
$ \PP{3k+2}^{\tau} =
\begin{cases}
\M{3k+2}^{A_1\times [\tau,\wl]} +\M{\blue{3k+1}}^{A_1\times [\tau,\wl]} +\M{3k}^{A_1\times [\tau,\wl]},
& \text{ if } \ell(\tau)>\ell(\tau 2),
\\
\M{3k+2}^{A_1\times [\tau,\wl]} +\M{\blue{3k+1}}^{A_1\times [\tau,\wl]/\langle s_2\rangle},
& \text{ if } \ell(\tau)<\ell(\tau 2),
\qquad (k\ge 1).
\end{cases}
$
\end{enumerate}
\end{thm}
Theorems~\ref{thm:k=1+sigma:P} and \ref{thm:0:k=2+sigma:P} provide character formulae for all projective covers in the block $\Bl_k$, for $k\ge 2$, and for some projectives in the blocks $\Bl_0$ and $\Bl_1$.

\subsection{Formulae for projective modules in $\Bl_1$}

Applying Theorem~\ref{thm:0:k=1:sigma} and the Soergel duality \eqref{tiltingD}, we obtain the following character formulae for the projective modules $\PP{0}^{\tau}$, for $\tau\in W_2$, in the block $\Bl_1$.

\begin{thm}  [Block $\Bl_1$]
   \label{thm:0:k=1:sigma:P}
The following formulae hold for projective modules in the block $\Bl_1$: for $\tau \in W_2$,
\[
\PP{0}^{\tau} =
\begin{cases}
\M{0}^{A_1\times \wl}+ \M{0}^{0\wl 2}+ \M{\red{1}}^{0\wl} +\M{2}^{0\wl},
& \quad  \text{ if }  \tau =\wl,
\\
\M{0}^{A_1\times[\tau,\wl]} +\M{0}^{0[\tau,\wl]2} +\M{{\red{1}}}^{0[\tau,\wl]/\langle s_1\rangle},
& \quad  \text{ if } \ell(\tau) >\ell(\tau 2), \; \tau\not= \wl,
\\
\M{0}^{A_1\times [\tau,\wl]},
& \quad \text{ if } \ell(\tau) <\ell(\tau 2).
\end{cases}
\]
\end{thm}

Together with Theorems~\ref{thm:k=1+sigma:P} and \ref{thm:0:k=2+sigma:P} (for $k=1$), Theorem~ \ref{thm:0:k=1:sigma:P} provides character formulae for all projective covers in the block $\Bl_1$.

\subsection{Formulae for projective modules in  $\Bl_0$}

Applying Theorem~\ref{thm:k=0:sigma} and the Soergel duality \eqref{tiltingD}, we obtain the following character formulae for the projective modules $\PP{\red{0}}^{0\tau}$, for $\tau\in W_2$, in the block $\Bl_0$.
\begin{thm}  [Block $\Bl_0$]
   \label{thm:k=0:sigma:P}
The following formulae hold for projective modules in the block $\Bl_0$: for $\tau \in W_2$ such that $\ell(\tau) <\ell(\tau 1)$,
\begin{align*}
\PP{\red{0}}^{0\tau}=\PP{\red{0}}^{0\tau 1}
&=  \M{\red{0}}^{0[\tau,\wl]/\langle s_1\rangle} +\M{\blue{1}}^{0[\tau,\wl]} +\M{2}^{0[\tau,\wl]}, \text{ for }\tau\not=e;
 \\
\PP{\red{0}}^{0} =\PP{\red{0}}^{01}
&=  \M{\red{0}}^{0[e,\wl]/\langle s_1\rangle} +\M{\blue{1}}^{0[e,\wl]/\langle s_2\rangle}.
\end{align*}
\end{thm}

Applying Theorem~\ref{thm:0:k=0:sigma} and Soergel duality \eqref{tiltingD}, we obtain the following character formulae for the projective modules $\PP{\red{0}}^{\tau}$, for $\tau\in W_2$, in the block $\Bl_0$.

\begin{thm} [Block $\Bl_0$]
   \label{thm:0:k=0:sigma:P}
The following formulae hold for projective modules in $\Bl_0$: for $\tau\in W_2$,
\begin{enumerate}
\item
Let $\tau=12121, 2121, 121$. Suppose $\tau =ij\tau'$ with $\ell(\tau)=\ell(\tau')+2$. Then
\begin{align*}
\PP{\red{0}}^{\tau 1} &= \PP{{\red{0}}}^{\tau}
 = \M{\red{0}}^{A_1 \times [\tau,\wl]/\langle s_1\rangle} + \M{\red{0}}^{0[\tau,\wl]\cup\{0\tau',0j\tau' \}} + \M{\blue{1}}^{0[\tau 1,\wl]}.
 \end{align*}
%
 Moreover we have
 \begin{align*}
\PP{\red{0}}^{21212} &= \PP{\red{0}}^{\wl} = \M{\red{0}}^{\{\wl,0\wl\}} + \M{\red{0}}^{\{02121, 012121\}} + \M{\blue{1}}^{\{021212, 0\wl\}},
\\
 \PP{\red{0}}^{2} &= \PP{\red{0}}^{21}
 = \M{\red{0}}^{A_1 \times [21,\wl]/\langle s_1\rangle} + \M{\red{0}}^{0[2121,\wl]/\langle s_1\rangle\cup\{01\}} + \M{\blue{1}}^{0[212,\wl]/\langle s_2\rangle},
 \\
\PP{\red{0}}^{e} &=\PP{\red{0}}^{1}=\M{\red{0}}^{A_1\times W_2/\langle s_1\rangle}.
 \end{align*}

\item
$
\PP{\blue{1}}^{\tau 2}
= \PP{\blue{1}}^{\tau} =
\begin{cases}
\M{\blue{1}}^{A_1\times [\tau,\wl]/\langle s_2\rangle} +\M{\red{0}}^{A_1\times [\tau,\wl]},
& \text{ if }  \ell(\tau)<\ell(\tau 2), \; \tau \not=e,
\\
\M{\blue{1}}^{A_1\times [e,w_0]/\langle s_2\rangle} + \M{\red{0}}^{A_1\times [e,w_0]/\langle s_1\rangle},
& \text{ if } \tau=e.
\end{cases}
$
\item
$\PP{2}^{\tau} =
\begin{cases}
\M{2}^{A_1\times [\tau,\wl]} +\M{\blue{1}}^{A_1\times [\tau,\wl]} +\M{{\red{0}}}^{A_1\times {[\tau,\wl]/\langle s_1\rangle}},
&   \text{ if } \ell(\tau) >\ell(\tau 2), \; \tau\not=\wl,
\\
\M{2}^{A_1\times [\tau,\wl]} +\M{\blue{1}}^{A_1\times {[\tau,\wl]/\langle s_2\rangle}},
&  \text{ if } \ell(\tau) <\ell(\tau 2).
\end{cases}
$
\end{enumerate}
\end{thm}
Together with Theorems~\ref{thm:k=1+sigma:P} and \ref{thm:0:k=2+sigma:P} (for $k=0$), Theorems~\ref{thm:k=0:sigma:P} and \ref{thm:0:k=0:sigma:P} provide character formulae for all projective covers (except $\PP{2}^{\wl}$) in the block $\Bl_0$.
%

\subsection{Projective injective modules in $\OO$}

Clearly an indecomposable module in $\OO$ is projective and tilting if and only if it is projective and injective.  The projective tilting modules in typical blocks of $\OO$ are well understood thanks to Gorelik's equivalence (see Proposition~\ref{prop:typ-block}), and they are exactly the tilting modules of typical dominant integral weights.

It remains to  classify the projective tilting modules in atypical blocks $\Bl_k$, for $k\in \N$.

\begin{thm}
The set $\{\TT{n}^{0 \wl} \mid n\in \N \}$ forms a complete list of projective tilting modules in $\Bl_k$, for each $k\in \N$. More precisely, we have
\[
\TT{n}^{0 \wl} \cong \PP{n}^e, \qquad \text{ for } n \in \N.
\]
\end{thm}

\begin{proof}
We make the following simple observation from the tilting module character formulae in Theorems~\ref{thm:k=1+sigma}, \ref{thm:k=0:sigma}, \ref{thm:0:k=2+sigma},  \ref{thm:0:k=1:sigma}, \ref{thm:0:k=0:sigma} (and their counterparts for characters of projective modules in Theorems~\ref{thm:k=1+sigma:P},  \ref{thm:0:k=2+sigma:P},  \ref{thm:0:k=1:sigma:P}, \ref{thm:k=0:sigma:P}, \ref{thm:0:k=0:sigma:P}):

{\it Any tilting module in $\OO$ has a lowest Verma flag of the form $\M{n}^e$  while any projective module in $\OO$  has a highest Verma flag of the form $\M{n}^{0\wl}$.
}

Hence the only possible isomorphisms between a tilting module and a projective module are of the form $\TT{n}^{0\wl} \cong \PP{m}^{e}$, for some $n,m\in \N$; another quick inspection of the Verma flag formulae for tilting modules and projective modules shows that such identities hold on the character level if and only if $n=m$.

It remains to show that $\TT{n}^{0 \wl} \cong \PP{n}^e$, for all $n$ in each block $\Bl_k$, which is equivalent to showing that $\TT{n}^{0 \wl}$ is projective.

We first consider the ``generic" cases of $\TT{n}^{0 \wl}$, for $n\ge 2$ and $n\neq 3k+5$, in the block $\Bl_k$ for any $k$.  The standard initial tilting module $\TT{f_n^{0\wl}-2\delta}$ to which we apply a translation functor $\E$ to obtain $\TT{n}^{0\wl}$ is typical dominant, and thus it is projective by Proposition~\ref{prop:typ-block}. Since $\E$ is an exact functor, the resulting module $\TT{n}^{0 \wl} =\E \TT{f_n^{0\wl}-2\delta}$ must be projective too, and hence must be isomorphic to $\PP{n}^{e}$ for character reason (by a case-by-case inspection).

For $\TT{3k+5}^{0 \wl}$, we apply a translation functor $\E$ to the standard initial tilting module $\TT{f_{3k+5}^{0 \wl}-2\delta}$. While ${f_{3k+5}^{0 \wl}-2\delta}$ is atypical, ${f_{3k+5}^{0 \wl}-4\delta}$ is typical and dominant and hence $\TT{f_{3k+5}^{0 \wl}-2\delta} \cong \PP{f_{3k+5}^{e}+2\delta}$ by the generic case above. Hence as a summand of a direct sum of projective modules $\E \TT{f_{3k+5}^{0 \wl}-2\delta}$, the tilting module $\TT{3k+5}^{0 \wl}$ must be projective , and hence must be isomorphic to $\PP{3k+5}^{e}$ for character reason.

Now we consider the case of $n=1$. We first examine the blocks $\Bl_k$, for $k\ge 1$. In this case we apply a translation functor to the tilting module $T_{f_1^{0\wl}-\delta+\ep_3}$, which is a typical projective tilting module. We obtain a Verma flag for a possibly direct sum of tilting modules such that the highest term in this Verma flag is $M^{0\wl}_1$. This implies that $T^{0\wl}_1$ is a summand of a projective module and hence is projective in the blocks  $\Bl_k$, for $k\ge 1$. For block $\Bl_0$ we apply a translation functor to the tilting module $T_{f_1^{0\wl}-\ep_2+\ep_3}$, which is a typical projective tilting module (note ${f_1^{0\wl}-\ep_2+\ep_3}=[3/2|0,0,0]$). Again, we verify that the highest term in the resulting Verma flag is indeed $M^{0\wl}_1$ (with multiplicity). This then implies that $T^{0\wl}_1$ in the block $\Bl_0$ is projective.

Finally, we consider the case of $n=0$ and $k\ge 0$. We consider the tilting module of highest weight $f_0^{\wl}+\delta+\ep_3$, which is typical and projective. We apply a translation functor to this module and again verify that the highest term in the resulting Verma flag is indeed $M^{0\wl}_0$ (with multiplicity). This then implies that $T^{0\wl}_0$ is projective.

The theorem is proved.
\end{proof}

\section{Jordan-H\"older multiplicities of Verma modules in $\OO$}
  \label{sec:comp}

In this section, we completely describe the Jordan-H\"older multiplicities of Verma modules in  blocks $\Bl_k$, for $k\ge 1$.

\subsection{Jordan-H\"older multiplicities for Verma modules in $\Bl_k$}

For any subset $D \subseteq W_2$, we introduce a shorthand notation
\[
\LL{n}^D =\sum_{\tau \in D} \LL{n}^\tau.
\]
We shall write the composition multiplicity formula for a Verma module as a sum of $\LL{g}$'s.
Via the BGG reciprocity, we convert the formulae for Verma flags of projective modules in Theorems~\ref{thm:k=1+sigma:P} and  \ref{thm:0:k=2+sigma:P} to the following formulae for composition multiplicities.

\begin{thm}
  \label{thm:k=1+sigma:C}
The following formulae hold for Jordan-H\"older multiplicities of Verma modules in the block $\Bl_k$, for $k\ge 1$ (except in Case~(1) where $k\ge 2$): for $\sigma \in W_2$,
\begin{enumerate}
\item
$\M{0}^{\sigma} =
\LL{0}^{[e,\sigma]}  +\LL{1}^{[e,\sigma]},
  \quad (k\ge 2).
$
\item
$\M{k-1}^{\sigma} =
\LL{k-1}^{[e,\sigma]} +\LL{\red{k}}^{[e,\sigma]/\langle s_1\rangle} +\sum\limits_{\stackrel{\ell(\tau) >\ell(\tau 1)}{\tau\le \sigma}} \LL{k+1}^\tau.
$
\item
$\M{\red{k}}^{\sigma 1} =\M{\red{k}}^{\sigma}
=  \LL{\red{k}}^{[e,\sigma]/\langle s_1\rangle} +  \LL{k+1}^{[e,\sigma]}+ \sum\limits_{\ell(\sigma)-2\ge \ell(\tau) > \ell(\tau 1)} \LL{k+1}^{\tau}, \quad
\text{ if } \ell(\sigma)> \ell(\sigma 1).
$
\item
$ \M{3k}^{\sigma} =
\LL{3k}^{[e,\sigma]} +\LL{\blue{3k+1}}^{[e,\sigma]/\langle s_2\rangle} +\sum\limits_{\stackrel{\ell(\tau) >\ell(\tau 2)}{\tau\le \sigma}}\LL{3k+2}^{\tau}.
$
\item
$\M{\blue{3k+1}}^{\sigma 2} =\M{\blue{3k+1}}^{\sigma}
=  \LL{\blue{3k+1}}^{[e,\sigma]/\langle s_2\rangle} +  \LL{3k+2}^{[e,\sigma]} + \sum\limits_{\ell(\sigma)-2\ge \ell(\tau) > \ell(\tau 2)}  \LL{3k+2}^{\tau}, \;
\text{ if } \ell(\sigma) > \ell(\sigma 2).
$
\item
$\M{n}^{\sigma} =  \LL{n}^{[e,\sigma]}  +\LL{n+1}^{[e,\sigma]},\quad \forall n \in \N \backslash \{0, k-1, k, 3k, 3k+1\}.
$
\end{enumerate}
\end{thm}

\begin{thm}
  \label{thm:0:k=2+sigma:C}
The following formulae hold for Jordan-H\"older multiplicities of Verma modules in the block $\Bl_k$ (where $k\ge 2$ in Cases (1)-(4) and $k\ge 1$ in Cases (5)-(8)): for $\sigma \in W_2$,
\begin{enumerate}
\item
$ \M{0}^{0\sigma} =
\LL{0}^{A_1\times [e,\sigma]} +\LL{1}^{[e,\sigma]}. 
$
\item
$ \M{k-1}^{0\sigma} =
\LL{k-1}^{A_1\times [e,\sigma]}  +\LL{k-2}^{0 [e,\sigma]} +\LL{\red{k}}^{[e,\sigma]/\langle s_1\rangle} +\sum\limits_{\stackrel{\ell(\tau)>\ell(\tau 1)}{\tau\le \sigma}} \LL{k+1}^{\tau}. 
$
\item
$ \M{\red{k}}^{0\sigma 1} =\M{\red{k}}^{0\sigma}
=  \LL{\red{k}}^{A_1\times [e,\sigma]/\langle s_1\rangle} +\LL{k-1}^{0 [e,\sigma]}  + \sum\limits_{\ell(\sigma)-2\ge \ell(\tau) > \ell(\tau 1)} \LL{k-1}^{0\tau}
\\
{}\qquad\qquad\qquad\quad
+\LL{k+1}^{[e,\sigma]}+ \sum\limits_{\ell(\sigma)-2\ge \ell(\tau) > \ell(\tau 1)} \LL{k+1}^{\tau},
\quad \text{ if } \ell(\sigma)>\ell(\sigma 1). 
$
\item
$ \M{k+1}^{0\sigma} =
\LL{k+1}^{A_1\times [e,\sigma]} +\LL{\red{k}}^{0 [e,\sigma]/\langle s_1\rangle}  +\sum\limits_{\stackrel{\ell(\tau)>\ell(\tau 1)}{\tau\le \sigma}} \LL{k-1}^{0\tau} +\LL{{k+2}}^{[e,\sigma]}. 
$
\item
$ \M{3k}^{0\sigma} =
\LL{3k}^{A_1\times [e,\sigma]}  +\LL{3k-1}^{0 [e,\sigma]} +\LL{\blue{3k+1}}^{ [e,\sigma]/\langle s_2\rangle}  +\sum\limits_{\stackrel{\ell(\tau)>\ell(\tau 2)}{\tau\le \sigma}} \LL{3k+2}^{\tau}. 
$
\item
$ \M{\blue{3k+1}}^{0\sigma 2} =\M{\blue{3k+1}}^{0\sigma}
=  \LL{\blue{3k+1}}^{A_1\times [e,\sigma]/\langle s_2\rangle} +\LL{3k}^{0[e,\sigma]} + \sum\limits_{\ell(\sigma)-2\ge \ell(\tau) > \ell(\tau 2)} \LL{3k}^{0\tau}
\\
{}\qquad\qquad\qquad\quad\quad
+\LL{3k+2}^{[e,\sigma]} + \sum\limits_{\ell(\sigma)-2\ge \ell(\tau) > \ell(\tau 2)} \LL{3k+2}^{\tau}$,
\quad if $\ell(\sigma) >\ell(\sigma 2). 
$
\item
$ \M{3k+2}^{0\sigma} =
\LL{3k+2}^{A_1\times [e,\sigma]} +\LL{\blue{3k+1}}^{0[e,\sigma]/\langle s_2\rangle} +\sum\limits_{\stackrel{\ell(\tau)>\ell(\tau 2)}{\tau\le \sigma}} \LL{3k}^{0\tau} +\LL{3k+3}^{ [e,\sigma]}. 
$
  \item
$\M{n}^{0\sigma} =  \LL{n}^{A_1\times [e,\sigma]}  +\LL{n-1}^{0 [e,\sigma]}  +\LL{n+1}^{[e,\sigma]},
\quad \forall n \in \N \backslash \{0, k-1, \red{k}, k+1, 3k, \blue{3k+1}, 3k+2\}. 
$
\end{enumerate}
\end{thm}
Theorems~\ref{thm:k=1+sigma:C} and \ref{thm:0:k=2+sigma:C} provide a complete description of composition multiplicities of all Verma modules in blocks $\Bl_k$, for $k\ge 2$.

\subsection{Jordan-H\"older multiplicities for Verma modules in $\Bl_1$}

Via the BGG reciprocity, we convert the formulae for Verma flags of projective modules in Theorems~\ref{thm:k=1+sigma:P}, \ref{thm:0:k=2+sigma:P} and \ref{thm:0:k=1:sigma:P} to the following formulae for composition multiplicities of Verma modules in $\Bl_1$. Together with Theorems~\ref{thm:k=1+sigma:C} and \ref{thm:0:k=2+sigma:C} (for $k=1$), this provides a complete description of composition multiplicities of all Verma modules in the block $\Bl_1$.

\begin{thm}
  \label{thm:0:k=1:C}
The following formulae hold for Jordan-H\"older multiplicities of Verma modules in the block $\Bl_1$: for $\sigma \in W_2$,
\begin{enumerate}
\item
$\M{0}^{\sigma} =
\LL{0}^{[e,\sigma]}  +\LL{\red{1}}^{[e,\sigma]/\langle s_1\rangle}  + \sum\limits_{\stackrel{\ell(\tau)>\ell(\tau 1)}{\tau\le \sigma}} \LL{2}^{\tau}.
$
\item
$ \M{0}^{0\sigma} =
\LL{0}^{A_1\times [e,\sigma]}   +\sum\limits_{\stackrel{\ell(\tau)>\ell(\tau 2)}{\tau \in [e,\sigma 2]}
} \LL{0}^{\tau} +\LL{\red{1}}^{[e,\sigma]/\langle s_1\rangle} +\sum\limits_{\stackrel{\ell(\tau)>\ell(\tau 1)}{\tau\le \sigma}} \LL{2}^{\tau}.
$
\item
$ \M{\red{1}}^{0\sigma 1} =\M{\red{1}}^{0\sigma}
=  \LL{\red{1}}^{A_1\times [e,\sigma]/\langle s_1\rangle} +\LL{0}^{0 [e,\sigma]} + \sum\limits_{\ell(\sigma)-2\ge \ell(\tau) > \ell(\tau 1)} \LL{0}^{0\tau}
\\ {}\qquad\qquad\quad\quad
+\LL{2}^{[e,\sigma]}+ \sum\limits_{\ell(\sigma)-2\ge \ell(\tau) > \ell(\tau 1)} \LL{2}^{\tau}  +\sum\limits_{\stackrel{\ell(\tau)>\ell(\tau 2)}{\tau\le \sigma}} \LL{0}^{\tau},
\quad \text{ if } \ell(\sigma)>\ell(\sigma 1).
$
\item
$ \M{2}^{0\sigma} =
\LL{2}^{A_1\times [e,\sigma]} +\LL{\red{1}}^{0 [e,\sigma]/\langle s_1\rangle}  +\sum\limits_{\stackrel{\ell(\tau)>\ell(\tau 1)}{\tau\le \sigma}} \LL{0}^{0\tau} +\LL{{3}}^{[e,\sigma]} +\delta_{\sigma,\wl} \LL{0}^{\wl}.
$
\end{enumerate}
\end{thm}

\begin{rem}
For the block $\Bl_0$, we can write down the Jordan-H\"older multiplicity formulae of all Verma modules except $\M{3}^{0\wl}$ and several cases among $\M{n}^{0\sigma}$ ($n=0,1,2, \sigma \in W_2$). This is caused by the missing character for the tilting module $\TT{2}^{0}$; see Remark~\ref{rem:T20}.
\end{rem}


\end{document}